\documentclass[12pt]{amsart}
\usepackage{amssymb,amsmath}
\usepackage{geometry}    
\usepackage{mathrsfs}

\usepackage{vmargin}
\setmargrb{1in}{1in}{1in}{1in}

\newtheorem{theorem}{Theorem}[section]
\newtheorem{lemma}[theorem]{Lemma}
\newtheorem{proposition}[theorem]{Proposition}
\newtheorem{corollary}[theorem]{Corollary}
\theoremstyle{definition}
\newtheorem{definition}[theorem]{Definition}
\newtheorem{remark}[theorem]{Remark}

\newtheorem{conjecture}[theorem]{Conjecture}

\newcommand{\scp}[1]{\langle#1\rangle}

\begin{document}

\baselineskip=10pt

\title{On the rate of accumulation of $(\alpha \zeta^{n})_{n\geq 1}$ mod $1$ to $0$}

\author{ Johannes Schleischitz}

\begin{abstract}
In this paper we study the distribution of the sequence $(\alpha \zeta^{n})_{n\geq 1}$ mod $1$, where
$\alpha,\zeta$ are fixed positive real numbers,
with special focus on the accumulation point $0$.
For this purpose we introduce approximation constants 
$\underline{\sigma}(\alpha,\zeta),\overline{\sigma}(\alpha, \zeta)$ and study their properties
in dependence of $\alpha,\zeta$, distinguishing in particular the cases of Pisot numbers, algebraic 
non Pisot numbers and transcendental values of $\alpha$ as well as $\zeta$.
\end{abstract}														
														
\maketitle

{\footnotesize{ Supported by FWF grant P24828 \\  
Institute of Mathematics, Department of Integrative Biology, BOKU Wien, 1180, Vienna, Austria \\

Math subject classification: 11J71, 11J81, 11J82, 11K55 \\
key words: Pisot numbers, distribution mod 1, Diophantine approximation, transcendence theory}\\

\vspace{8mm}

\section{Introduction}  \label{fff}
This paper deals with the distribution of $\alpha \zeta^{n}$ mod $1$ for arbitrary but fixed positive 
real numbers $\alpha,\zeta$ as $n$ runs through the positive integers. We are in particular interested
in pairs $\alpha, \zeta$ for which rather fast accumulation to $0$ occurs, either for a sequence of arbitrarily
large values of $n$ or for all sufficiently large values of $n$. 
We will treat these two cases separately, measuring the rate of accumulation with approximation constants 
$\underline{\sigma}(\alpha,\zeta),\overline{\sigma}(\alpha,\zeta)$ we will introduce in section \ref{ggg}.
Related problems were first studied by Pisot in \cite{10}                                 
using methods of Fourier Analysis. An interesting result of Pisot states that if for some $\zeta>1$ the 
sequence $\alpha \zeta^{n}$ mod $1$ tends to zero as $n$ tends to infinity, so roughly speaking the numbers $\alpha \zeta^{n}$ 
somehow ''converge to integers'', then $\zeta$ must be algebraic and of a special shape, called Pisot numbers to his honors.
We will give a definition and known properties of Pisot numbers in the section \ref{ggg}.

In the present paper we don't make use of Fourier analysis, many given proofs rely on basic 
properties of symmetric polynomials or classical Diophantine approximation properties, in the latter case
mostly concerning the approximation of $n\cdot \zeta$ mod $1$ for fixed $\zeta\in{\mathbb{R}}$ as $n$ runs through the integers,
and higher dimensional generalizations. For this purpose we at first introduce 
some notation, some of which is classic notation and some invented for our special purpose.

\subsection{Basic facts and notations.}  \label{ggg}

At first a basic definition, whose parameter $x$ will later mostly be of the form $\alpha \zeta^{n}$.

\begin{definition}
For a real number $x\in{\mathbb{R}}$ denote with $\lfloor x\rfloor\in{\mathbb{Z}}$ the largest integer smaller or equal $x$,
$\lceil x\rceil\in{\mathbb{Z}}$ the smallest integer greater or equal $x$ and $\{x\}\in{[0,1)}$ the fractional part
of $x$, i.e. $\{x\}=x-\lfloor x\rfloor$. Furthermore denote with $\scp{x}\in{\mathbb{Z}}$ the closest integer to $x$ and with
 $\Vert x\Vert:=\vert x-\scp{x}\vert\in{[0,1/2]}$ the distance to the closest integer to $x$, 
with the special convention if $\{x\}=1/2$ then $\scp{x}:= \lfloor x\rfloor$.
So clearly we have $\scp{x}\in{\{\lfloor x\rfloor, \lceil x\rceil\}}$ and $\Vert x\Vert= \vert x-\scp{x}\vert$.
If for a sequence $(x_{n})_{n\geq 1}$ we have $\lim_{n\to\infty} \Vert x_{n}\Vert=0$, we will say $(x_{n})_{n\geq 1}$ 
converges to integers.
\end{definition}

By Bolzano-Weierstrass Theorem \cite{14}, an alternative characterization of convergence to integers is 
that the sequence $\{x_{n}\}$ can only have the accumulation point $\{0\}.$      

We will in general restrict to the case 
\begin{equation} \label{eq:husar}
\zeta>1, \qquad \qquad  \alpha>0,
\end{equation}
as for $\zeta\in{(-1,1)}$ we clearly have $\lim_{n\to\infty} \alpha \zeta^{n}=0$ for all $\alpha\in{\mathbb{R}}$ 
and $\zeta\mapsto -\zeta$, $\alpha \mapsto -\alpha$ does not affect the properties of $\alpha \zeta^{n}$ mod $1$,
and the special cases $\zeta\in{\{-1,1\}}$ or $\alpha=0$ are of no interest either. 
  
We are particularly interested in $\alpha, \zeta$, for which at least for a subset $(n_{1},n_{2},\ldots)$ of positive
integers, the values $\alpha \zeta^{n_{i}}$ converge to integers rather quickly.

In order to measure this convergence and more general the distribution of $\alpha \zeta^{n}$ mod $1$ in dependence of $n$,
we now introduce  
\begin{equation} \label{eq:uchazi}
 \sigma_{n}(\alpha, \zeta):= -\frac{\log \Vert \alpha \zeta^{n} \Vert}{\log (\alpha\zeta^{n})}
\end{equation}
such as the derived approximation constants
\begin{equation} \label{eq:mozart}
 \underline{\sigma}(\alpha,\zeta):= \liminf_{n\to\infty} \sigma_{n}(\alpha,\zeta),
\qquad \overline{\sigma}(\alpha,\zeta):= \limsup_{n\to\infty} \sigma_{n}(\alpha,\zeta).
\end{equation}
Large values of $\overline{\sigma}(\alpha,\zeta)$ mean that for some sequence $(n_{1},n_{2},\ldots)$ of positive integers 
which tends monotonically to infinity, the values $\Vert \alpha \zeta^{n_{i}}\Vert$ converge to integers very fast.
 In particular, $\overline{\sigma}(\alpha,\zeta)>0$ gives an exponential convergence to integers of the sequence.

Similarly, large values of $\underline{\sigma}(\alpha,\zeta)$ give fast convergence of the sequence 
$(\alpha \zeta^{n})_{n\geq 1}$ to integers,
and in particular $\underline{\sigma}(\alpha,\zeta)>0$ gives exponential convergence. 

In case of $\alpha \neq 0$ and $\zeta\in{(-1,1)}$, it is easy to see that 
\begin{equation} \label{eq:portugues}
\underline{\sigma}(\alpha,\zeta)=\overline{\sigma}(\alpha,\zeta)=-1, \qquad \zeta\in{(-1,1)}.
\end{equation}
So in the sequel assume $\zeta>1, \alpha>0$. 
In this case we have $\lim_{n\to\infty} \alpha \zeta^{n}=\infty$ as well as
$0\leq \Vert \alpha\zeta^{n}\Vert\leq 1/2$ for all $n$, so clearly
\begin{equation} \label{eq:joehklar}
 \overline{\sigma}(\alpha,\zeta)\geq \underline{\sigma}(\alpha,\zeta)\geq 0, \qquad \zeta>1.
\end{equation}
Note that the expressions $\underline{\sigma}(\alpha,\zeta),\overline{\sigma}(\alpha,\zeta)$
can be written in the easier form
\begin{equation} \label{eq:hauser}
 \underline{\sigma}(\alpha,\zeta)= \liminf_{n\to\infty} -\frac{\log \Vert \alpha \zeta^{n}\Vert}{n\cdot \log \zeta},\qquad \overline{\sigma}(\alpha,\zeta):= \limsup_{n\to\infty} -\frac{\log \Vert \alpha \zeta^{n}\Vert}{n\cdot \log \zeta}.
\end{equation}
This can easily be deduced by the definition of the quantities, as for sequences $(x_{n})_{n\geq 1},(y_{n})_{n\geq 1}$
 with $\lim_{n\to\infty} y_{n}=\infty$ as well as $\lim_{n\to\infty} \frac{x_{n}}{y_{n}}=:Z$ and fixed 
$\delta\in{\mathbb{R}}$ we have 
\begin{equation} \label{eq:harz}
 \lim_{n\to\infty} \frac{x_{n}}{y_{n}+\delta}= \lim_{n\to\infty} \frac{x_{n}}{y_{n}} \cdot \frac{y_{n}}{y_{n}+\delta}=Z\cdot 1=Z.
\end{equation}
Applying (\ref{eq:harz}) to $x_{n}:= - \log \Vert \alpha \zeta^{n}\Vert, y_{n}:= \log \zeta^{n}= n\log \zeta$ 
and $\delta= \log \alpha$, which satisfy the conditions by (\ref{eq:husar}),
 and recalling (\ref{eq:mozart}) yields (\ref{eq:hauser}).

Also note that in the case $\overline{\sigma}(\alpha,\zeta)=0$ respectively $\underline{\sigma}(\alpha,\zeta)=0$, one cannot decide if the sequence $\alpha \zeta^{n}$ respectively some subsequence $\alpha \zeta^{n_{i}}$ tends to integers.

An easy property of the quantities $\sigma_{n}(\alpha,\zeta)$ is
 $\sigma_{n}(\alpha,\zeta^{k})= \sigma_{nk}(\alpha,\zeta)$ for all $\alpha,\zeta$ and
 $k=1,2,3,\ldots$ and thus taking limits
\begin{eqnarray}
 \underline{\sigma}(\alpha,\zeta^{k})&\geq&\underline{\sigma}(\alpha,\zeta), \qquad k=1,2,3,\ldots  \label{eq:nanno}        \\
 \overline{\sigma}(\alpha,\zeta^{k})&\leq& \overline{\sigma}(\alpha,\zeta), \qquad k=1,2,3,\ldots  \label{eq:nanoleuchte}
\end{eqnarray}
holds. Another easy property of the quantities
$\underline{\sigma}(\alpha,\zeta),\overline{\sigma}(\alpha,\zeta)$ 
is given in the following Proposition which we will use in section \ref{subsek}.

\begin{proposition} \label{interessant}
Let $\alpha,\zeta>1$ be real numbers and $M,N>0$ integers. Then

\begin{eqnarray*}
 \underline{\sigma}(M\alpha,N\zeta)&\geq& \max\left(\frac{\log \zeta\cdot \underline{\sigma}(\alpha,\zeta)-\log N}{\log \zeta+\log N},
0\right)  \\
\overline{\sigma}(M\alpha,N\zeta)&\geq& \max\left(\frac{\log \zeta\cdot \overline{\sigma}(\alpha,\zeta)-\log N}{\log \zeta+\log N},
0\right).
\end{eqnarray*}

\end{proposition}
 
\begin{proof}
Without loss of generality we may assume $M,N,\alpha,\zeta$ all to be positive.
The bound $0$ is the trivial bound from (\ref{eq:joehklar}), so we may
restrict to the case $\frac{\log N}{\log \zeta}<\underline{\sigma}(\alpha,\zeta)$
 resp. $\frac{\log N}{\log \zeta}<\overline{\sigma}(\alpha,\zeta)$,
or equivalently
\begin{equation} \label{eq:auchgut}
 N<\zeta^{-\underline{\sigma}(\alpha,\zeta)}, \qquad \text{resp.} \quad  N<\zeta^{-\overline{\sigma}(\alpha,\zeta)}.
\end{equation}
For any positive integers $M,N$ we have
\begin{eqnarray}
\left\Vert (M\alpha)(N\zeta)^{n}\right\Vert &=&\left\vert\scp{(M\alpha)(N\zeta)^{n}}-(M\alpha)(N\zeta)^{n}\right\vert  \nonumber \\
&=& \left\vert \scp{MN^{n}\alpha \zeta^{n}}-MN^{n}\alpha \zeta^{n}\right\vert    \nonumber \\
&=& \left\vert \scp{MN^{n}(\scp{\alpha \zeta^{n}}\pm \Vert \alpha \zeta^{n}\Vert)}-MN^{n}\alpha \zeta^{n}\right\vert   \nonumber \\
&=& \left\vert MN^{n}\scp{\alpha \zeta^{n}}\pm \scp{MN^{n}\Vert \alpha \zeta^{n}\Vert}-MN^{n}\alpha \zeta^{n}\right\vert  \nonumber \\
&=& \left\vert MN^{n}(\scp{\alpha \zeta^{n}}-\alpha b^{n})\pm \scp{MN^{n}\left\Vert \alpha \zeta^{n}\right\Vert}\right\vert   \nonumber \\
&=& \left\vert \pm MN^{n}\Vert \alpha \zeta^{n}\Vert\pm \scp{MN^{n}\left\Vert \alpha \zeta^{n}\right\Vert}\right\vert.\label{eq:starlight} 
\end{eqnarray}
By our restrictions (\ref{eq:auchgut}), for
any sufficiently large $n\geq n_{0}$ resp. for arbitrarily large values of $n\geq n_{0}$ 
we have $MN^{n}\Vert \alpha \zeta^{n}\Vert< \frac{1}{2}$, which is equivalent to $\scp{MN^{n}\Vert \alpha \zeta^{n}\Vert}=0$.
 In view of (\ref{eq:starlight}) this yields
\[
 \left\Vert (M\alpha)(N\zeta)^{n}\right\Vert \leq MN^{n}\Vert \alpha \zeta^{n}\Vert
\]
for the respective values $n$. 
Taking logarithms according to (\ref{eq:uchazi}) yields for any $\epsilon>0$ 
and the respective values of $n$ (restricting to $n\geq n_{1}=n_{1}(\epsilon)$
for some $n_{1}(\epsilon)>n_{0}$ if needed)
\[
 \sigma_{n}(M\alpha,N\zeta)\geq \frac{(\sigma_{n}(\alpha,\zeta)-\epsilon)\log\zeta-\log N}{\log \zeta+\log N}.
\]
The assertion follows with $\epsilon\to 0$ by the definition of the quantities
 $\underline{\sigma}(\alpha,\zeta),\overline{\sigma}(\alpha,\zeta)$. 
\end{proof}

\begin{remark}
 Note that in case of $\alpha \neq 0$ and $\zeta\in{(0,1)}$ we have
\begin{eqnarray*}
 N\zeta &>& 1 \quad \Longrightarrow \quad \underline{\sigma}(M\alpha,N\zeta)\geq 0  \\
 N\zeta &<& 1 \quad \Longrightarrow \quad \underline{\sigma}(M\alpha,N\zeta)=\overline{\sigma}(M\alpha,N\zeta) = -1.
\end{eqnarray*}
This is easily deduced by (\ref{eq:portugues}) and (\ref{eq:joehklar}).
\end{remark}

We will now give the definition of a class of algebraic numbers with
a highly non-generic and thus interesting behavior concerning the sequence
$(\alpha\zeta^{n})_{n\geq 1}$ mod $1$. 

\begin{definition}{(Pisot numbers, Pisot polynomials, Pisot units)}
 A real algebraic integer $\zeta>1$ is called {\em Pisot number}, 
if all its conjugates lie strictly inside the unit circle of the complex plane.
 If a Pisot number is a unit in the ring of algebraic integers, we will call it a {\em Pisot unit}.
We will refer to the monic minimal polynomial $P\in{\mathbb{Z}[X]}$ of a Pisot number $\zeta$
 as the {\em Pisot polynomial of $\zeta$}.
In general call a polynomial a Pisot polynomial if it is the Pisot polynomial of a Pisot number $\zeta$,
and the unique root greater than $1$ of a Pisot polynomial $P$ the {\em Pisot number associated to $P$}.
\end{definition}

We will sum up known results for Pisot numbers we will refer to in the sequel,
transferred into our notation, in the following Theorem
(which we will call Pisot Theorem although the results may not be entirely due to him).

\begin{theorem}[Pisot]  \label{pisot}
Pisot numbers have the property $\underline{\sigma}(1,\zeta)>0$,
 i.e. $\zeta^{n}$ converges to integers at exponential rate.
This property characterizes Pisot numbers among all real algebraic numbers.
Even the following stronger assertion holds:
if $\alpha\zeta^{n}$ tends to integers for a real algebraic number $\zeta>1$,
then $\zeta$ is a Pisot number and $\alpha\in{\mathbb{Q}(\zeta)}$,
where $\alpha=1$ is always a possible choice.

Moreover, if for any real $\zeta>1$ the sequence 
$(\Vert\alpha \zeta^{n}\Vert)_{n\geq 1}$ is square-summable, 
then $\zeta$ is a Pisot number {\upshape(}and clearly again $\alpha\in{\mathbb{Q}(\zeta)}$,
and $\alpha=1$ is always a possible choice.{\upshape)}
\end{theorem}

The first assertion is easily seen by looking at the sum of the powers $\sum_{j=1}^{k} \zeta^{n}$
 of the conjugates $\zeta_{1}=\zeta, \zeta_{2},\ldots,\zeta_{k}$ of $\zeta$, where $k$
 denotes the degree $[\mathbb{Q}(\zeta):\mathbb{Q}]$ of $\zeta$.
 Every such sum is an integer as it is a symmetric polynomial in
 $\zeta_{1},\zeta_{2},\ldots ,\zeta_{k}$. We will recall a detailed proof
in Theorem \ref{zauber}. See \cite{10} or chapter 5 in \cite{101} for the proofs of the remaining
and slightly refined results. 
At this point it should be mentioned that there are only countably
 many $\zeta$ such that $\alpha \zeta^{n}$ converges to integers
for some auxiliary $\alpha$, an immediate consequence of Theorem 5.6.1 in \cite{101},
but the question if any such $\zeta$ is transcendental is open.																																																																			
Theorem \ref{pisot} immediately yields

\begin{theorem}  \label{gans}
 Let $\zeta>1$ be a real number but not a Pisot number.
 Then for any $\alpha\in{\mathbb{R}}$ we have $\underline{\sigma}(\alpha,\zeta)=0$.
If $\zeta$ is a Pisot number and $\alpha\notin{\mathbb{Q}(\zeta)}$, we have 
$\underline{\sigma}(\alpha,\zeta)=0$ as well.
\end{theorem}

\begin{proof}
 If otherwise $\Vert \alpha \zeta^{n}\Vert\leq \zeta^{-n\epsilon}$ for some $\epsilon>0$
and all $n\geq n_{0}$ sufficiently large, then 
\[
\sum_{n=1}^{\infty} \Vert \alpha \zeta^{n}\Vert^{2}\leq 
\sum_{n=1}^{n_{0}-1}\Vert \alpha \zeta^{n}\Vert^{2}+\sum_{n=n_{0}}^{\infty} \zeta^{-2n\epsilon}
\leq n_{0}+ \frac{1}{1-\zeta^{-2\epsilon}}<\infty.
\]
This contradicts the fact that only Pisot numbers have this property by Theorem \ref{pisot}.
The proof of the second assertion is similar due to facts from Theorem \ref{pisot}. 
\end{proof}         
                                                                         
We will discuss properties of Pisot numbers concerning 
the approximation constants $\underline{\sigma}(1,\zeta),\overline{\sigma}(1,\zeta)$ in more detail 
in section \ref{hhh}. Now we will only give one more well known basic fact

\begin{proposition} \label{klaro}
Any Pisot polynomial $P$ is irreducible. 
\end{proposition}

\begin{proof}
Clearly, the constant coefficient of $P$ is a nonzero integer.
Consequently, if $P=Q\cdot R$, with $Q,R$ non-constant polynomials,
the constant coefficients of $Q,R$ have this property too, so their 
absolute values are at least $1$.
Hence both $Q,R$ must have at least one root of absolute value larger than $1$,
since by Vieta Theorem \ref{vieta} the product of
the roots of $Q$ resp. $R$ are just the constant coefficient of $Q$ resp. $R$.
A contradiction to the fact that there is only one root of $P$ with absolute value 
greater or equal than $1$. 
\end{proof}

As indicated in the introduction section \ref{fff}, our approach will at some places deal with a classic 
simultaneous Diophantine approximation problem. 

\begin{definition} \label{defibrillator}
For a positive inter $d$ and 
$\boldsymbol{\zeta}:=(\zeta_{1},\zeta_{2},\ldots, \zeta_{d})\in{\mathbb{R}^{d}}$
define by $\lambda_{d}(\boldsymbol{\zeta})$ respectively $\widehat{\lambda}_{d}(\boldsymbol{\zeta})$ 
the supremum of all $\mu\in{\mathbb{R}}$ such that
\begin{eqnarray}  
 \left\vert x\right\vert &\leq& X    \nonumber \\
 \max_{1\leq k\leq d} \left\vert x\zeta_{i}-y_{i}\right\vert &\leq& X^{-\mu}  \label{eq:vir}
\end{eqnarray}
has a solution $(x,y_{1},\ldots,y_{d})\in{\mathbb{Z}^{d+1}}$ for some arbitrarily large values of $X$ respectively
 all sufficiently large values of $X$.
\end{definition}

By Minkowski's lattice point Theorem, for 
all $\boldsymbol{\zeta}\in{\mathbb{R}^{d}}$ we have the well known result
\begin{equation}  \label{eq:hunt}
\lambda_{d}(\boldsymbol{\zeta})\geq \widehat{\lambda}_{d}(\boldsymbol{\zeta})\geq \frac{1}{d},
\end{equation}
see the first pages of \cite{3} for instance.                                                     
For almost all $\boldsymbol{\zeta}\in{\mathbb{R}^{d}}$, there is actually equality in 
both inequalities (\ref{eq:hunt}), see \cite{khin}.
For our purposes it suffices to restrict to the case $d=1$.

\begin{theorem}[Khinchin] \label{khini}
The set of $\zeta\in{\mathbb{R}}$ with $\lambda_{1}(\zeta)>1$ has Lebesgue measure $0$. 
\end{theorem} 

We will later need the following result by Davenport, Schmidt and Laurent, see \cite{1},\cite{2}.     

\begin{theorem}[Davenport, Laurent, Schmidt] \label{joxe}
 Let $\boldsymbol{\zeta}=(\zeta,\zeta^{2},\ldots ,\zeta^{d})$ for $\zeta\in{\mathbb{R}}$ not
algebraic of degree $\leq \left \lceil \frac{d}{2} \right \rceil$. Then

\[
 \widehat{\lambda}_{d}(\boldsymbol{\zeta})\leq \frac{1}{\left\lceil \frac{d}{2}\right\rceil}.
\]
\end{theorem}

We will also refer to the well-known Roth Theorem for algebraic numbers
which we will state in our notation.

\begin{theorem}[Roth] \label{roth}
For an irrational, algebraic real number $\zeta$ we have $\lambda_{1}(\zeta)=1$.
\end{theorem}

We will later refer to another type of classical approximation constants too that deal with linear forms
and that are somehow dual to $\lambda_{d}$ and $\widehat{\lambda}_{d}$. 

\begin{definition}
For a real vector
 $\boldsymbol{\zeta}=(\zeta_{1},\zeta_{2},\ldots \zeta_{d})$ let the quantities 
$w_{d}(\boldsymbol{\zeta})$ resp. $\widehat{w}_{d}(\boldsymbol{\zeta})$
be given by the supreme of all $\nu\in{\mathbb{R}}$ such that
\begin{equation}  \label{eq:rotzz}
 \vert x+y_{1}\zeta_{1}+y_{2}\zeta_{2}+\cdots +y_{d}\zeta_{d}\vert \leq H^{-\nu}
\end{equation}
has a solution $(x,y_{1},\ldots,y_{d})\in{\mathbb{Z}^{d+1}}$  
with $\max(\vert x\vert,\vert y_{1}\vert,\ldots,\vert y_{d}\vert)\leq H$
for arbitrarily large resp. all sufficiently large values $H$.
\end{definition}

A connection between the values of $\lambda_{d}$ and $w_{d}$ is given by

\begin{theorem}[Khinchin]  \label{khinchin}
For any $\boldsymbol{\zeta}\in{\mathbb{R}^{d}}$ we have

\[
  w_{d}(\boldsymbol{\zeta})\geq d\lambda_{d}(\boldsymbol{\zeta})+d-1, \qquad 
	\lambda_{d}(\boldsymbol{\zeta})\geq \frac{w_{d}(\boldsymbol{\zeta})}{(d-1)w_{d}(\boldsymbol{\zeta})+d}.
\]
\end{theorem}

Finally a definition about polynomials we will use frequently in section \ref{sekieren}.

\begin{definition}
 For a polynomial $P(X)=a_{k}X^{k}+a_{k-1}X^{k-1}+\cdots +a_{0}$ with integer coefficients
let $H(P):=\max_{0\leq j\leq k} \vert a_{k}\vert$. For an algebraic number $z$ put $H(z):=H(P)$
with the minimal polynomial with relatively prime coefficients $P\in{\mathbb{Z}[X]}$ of $z$. 
\end{definition}

Note that the quantities $w_{d}(\boldsymbol{\zeta}),\widehat{w}_{d}(\boldsymbol{\zeta})$
can be equivalently defined as
\begin{equation} \label{eq:schroedingerkatz}
 w_{d}(\boldsymbol{\zeta})=  \limsup_{H\to\infty}\max_{\Vert L\Vert_{\infty}\leq H} -\frac{\log L(\boldsymbol{\zeta})}{\log H},\quad 
 \widehat{w}_{d}(\boldsymbol{\zeta})= \liminf_{H\to\infty}\max_{\Vert L\Vert_{\infty}\leq H}-\frac{\log L(\boldsymbol{\zeta})}{\log H},
\end{equation}
where $L(X_{1},X_{2},\ldots,X_{d})=a_{0}+a_{1}X_{1}+a_{2}X_{2}+\cdots+a_{d}X_{d}$ for $a_{j}\in{\mathbb{Z}}$ and
$\Vert L\Vert_{\infty}:=\max_{0\leq j\leq d} \vert a_{j}\vert$. So for any $H>0$
the maxima are taken among all integral linear forms $L$ with coefficients bounded by $H$.
Sprindzuk \cite{31} works with a similar notation for example.                             
The expressions in (\ref{eq:schroedingerkatz}) are in notable conformity to the possible definition of 
the quantities $\underline{\sigma}(\alpha,\zeta),\overline{\sigma}(\alpha,\zeta)$ in (\ref{eq:hauser}), 
which again underlines that it is pretty natural to consider these quantities. 

\section{Results for $\underline{\sigma}(\alpha,\zeta),\overline{\sigma}(\alpha,\zeta)$ with algebraic $\alpha,\zeta$} \label{sekieren}

\subsection{Pisot numbers}   \label{hhh}

Pisot discovered in that the sequence of positive integer powers of Pisot numbers converge exponentially to integers, 	 
i.e. $\underline{\sigma}(1,\zeta) > 0$. We want to prove some more detailed results in terms of the quantities
 $\underline{\sigma}(1,\zeta),\overline{\sigma}(1,\zeta)$. For preparation we need the following well-known result.

\begin{definition} \label{defi}
A polynomial $P\in{\mathbb{Z}[X_{1},\ldots,X_{k}]}$ is called symmetric, if for all bijections (=permutations)
$\sigma:\{1,2,\ldots ,k\}\mapsto \{1,2,\ldots,k\}$ the polynomial remains unaffected.
\end{definition}

\begin{definition}  \label{einfach}
 The elementary symmetric polynomials in $k$ variables $X_{1},X_{2},\ldots ,X_{k}$ are given by

\[
 \mu_{k,1}:=\sum_{j=1}^{k} X_{j}, \quad \mu_{k,2}:=\sum_{1\leq i<j\leq k}X_{i}X_{j}, \quad \ldots, 
\quad \mu_{k,k}:=\prod_{1\leq j\leq k} X_{j}.
\]

\end{definition}

\begin{theorem}  \label{elemsatz}
 Every symmetric polynomial is a polynomial with integer coefficients in the elementary symmetric polynomials.
\end{theorem} 

See \cite{11} for a proof.
																										
\begin{theorem} [Vieta]\label{vieta}
Let $P=a_{k}X^{k}+a_{k-1}X^{k-1}+\cdots +a_{0}$ be a polynomial with integer coefficients and roots $\zeta_{1},\ldots,\zeta_{k}$
counted with multiplicity. Then $a_{j}=\frac{(-1)^{k-j}}{a_{k}} \mu_{k,k+1-j}$ with $\mu_{.,.}$ from the above Definition.
\end{theorem}

Now we are ready to present a first Theorem concerning the 
quantities $\underline{\sigma}(1,\zeta),\overline{\sigma}(1,\zeta)$.
																										
\begin{theorem}  \label{zauber}
 Let $\zeta$ be a Pisot number of degree $[\mathbb{Q}(\zeta):\mathbb{Q}]=k$. Then we have 	 
	 
\begin{eqnarray}
 0 &<& \underline{\sigma}(1,\zeta) \quad \leq \quad \frac{1}{k-1}   \label{eq:milana} \\
 0 &<& \underline{\sigma}(1,\zeta) \quad \leq \quad \overline{\sigma}(1,\zeta) \quad \leq \quad k-1   \label{eq:malina}
\end{eqnarray}
	
\end{theorem}

\begin{proof}
 Let $\zeta_{1}=\zeta,\zeta_{2},\ldots, \zeta_{k}$ be the conjugates of $\zeta$.
Note first, that for all positive integers $n$
\begin{equation}  \label{eq:mann}
 \sum_{j=1}^{k} \zeta_{j}^{n}\in{\mathbb{Z}}, \qquad n\geq 1.
\end{equation}
That is because it is a symmetric polynomial in the variables $\zeta_{j}$ with integer coefficients. 
By Theorem \ref{elemsatz} it is an integer linear combination of the elementary symmetric polynomials,
 which are itself integers as they are the coefficients of the 
minimal polynomial $P(X) := \prod_{j=1}^{k} (X-\zeta_{j})$ of $\zeta$ (observe $\zeta$ is an algebraic integer).
 Thus we have (\ref{eq:mann}).

We now first proof (\ref{eq:milana}). On the other hand, by definition all other roots 
$\zeta_{2},\zeta_{3},\ldots,\zeta_{k}$ of $P(X)$ apart from $\zeta$
have absolute value smaller than $1$, so let $0<f<1$ be the maximum absolute value among these 
and put $z := \frac{\log f}{\log \zeta}<0.$ Then for all positive integers $n$ we have
\begin{equation} \label{eq:fritzi}
 \sum_{j=2}^{k} \zeta_{j}^{n}\leq (k-1)\cdot f^{n}= (k-1)\cdot \zeta^{nz}.
\end{equation}
On the other hand, by (\ref{eq:mann}) and as the right hand side of
 (\ref{eq:fritzi}) converges to $0$ as $n\to\infty$ we have
\begin{equation} \label{eq:hansbirschl}
 \Vert \zeta^{n}\Vert= \sum_{j=2}^{k} \zeta_{j}^{n}, \qquad n\geq n_{0}.
\end{equation}
Combining (\ref{eq:hansbirschl}) with (\ref{eq:fritzi}) we have
\begin{equation*}
\underline{\sigma}(1,\zeta):= \liminf_{n\to\infty} -\frac{\log \Vert \zeta^{n}\Vert}{n\cdot \log \zeta}= 
\liminf_{n\to\infty}-\frac{\sum_{j=2}^{k} \zeta_{j}^{n}}{n\cdot \log \zeta}\geq -\liminf_{n\to\infty}\frac{\log\left((k-1)\zeta^{nz}\right)}{n\cdot\log \zeta} =-z> 0, 
\end{equation*}
the left hand side of (\ref{eq:milana}).

For the upper bound in (\ref{eq:milana}) first observe, that since the product of the roots
 is the constant coefficient of $P(X)$ which is a nonzero integer,
\begin{equation*}
 1\leq \prod_{j=1}^{k} \vert \zeta_{j}\vert= \zeta \prod_{j=2}^{k} \zeta_{j}\leq \zeta\cdot f^{k-1}
\end{equation*}
yields $f\leq \zeta^{-\frac{1}{k-1}}$ or equivalently 
\begin{equation} \label{eq:tom}
z\leq \frac{1}{k-1}.
\end{equation}
Now let $\psi_{1}=0,\psi_{2},\ldots,\psi_{k}$ be the angles 
of $\zeta_{k}$ in the complex plane in the interval $[0,2\pi)$
and put $\phi_{j}:=\frac{\psi_{j}}{2\pi}\in{[0,1)}$ for $1\leq j\leq k$.
In general denote by $\psi(t)$ the angle in $[0,2\pi)$
of a complex number $t$ and put $\phi(t):=\frac{\phi(t)}{2\pi}$.
 By the identity $\psi(t^{n})=n\cdot \psi(t)$ we have   
\begin{eqnarray}
 \psi(\zeta_{j}^{n})&=&n\cdot \psi_{j}, \qquad 1\leq j\leq k                \\
 \phi(\zeta_{j}^{n})&=& n\cdot \phi_{j}, \qquad 1\leq j\leq k. 
\end{eqnarray}
By (\ref{eq:hunt}) with $d=n$ there exist arbitrarily large values of $n$ such that
simultaneously 
\[
\Vert\phi(\zeta_{j}^{n})\Vert=\Vert n\psi_{j}\Vert \leq \frac{1}{n}<\frac{1}{8}\cdot 2\pi, \qquad 1\leq j\leq k
\]
for certain arbitrarily large values of $n$. (For the rest of of the proof of (\ref{eq:milana}) we
only consider such a sequence of values for $n$, and assume $n$ is sufficiently large.)
Hence $\psi(\zeta_{j}^{n})=2\pi\cdot \phi(\zeta_{j}^{n})$, where the equality 
is viewed mod $2\pi$, hence $\vert\phi(\zeta_{j}^{n})-2\pi\vert< \frac{\pi}{4}$.Thus
\[
\rm{Re}\left(\zeta_{j}^{n}\right) \geq \cos(\psi_{j})\left\vert \zeta_{j}^{n}\right\vert = \cos(\psi_{j})\left\vert \zeta_{j}\right\vert^{n}
\geq \frac{1}{\sqrt{2}}\cdot \left\vert \zeta_{j}\right\vert^{n}, \quad 1\leq j\leq k.
\]
Since $\sum_{j=1}^{k} \zeta_{j}^{n}$ is an integer and $\zeta^{n}$ is real, 
$\sum_{j=2}^{k} \zeta_{j}^{n}$ is real. Hence and recalling (\ref{eq:hansbirschl}) gives
\[
\Vert \zeta^{n}\Vert= \sum_{j=2}^{k} \zeta_{j}^{n}= \rm{Re}\left(\sum_{j=2}^{k} \zeta_{j}^{n}\right)
\geq \frac{1}{\sqrt{2}}\cdot \sum_{j=2}^{k} \vert \zeta_{j}\vert^{n}
\geq  \frac{1}{\sqrt{2}}\cdot f^{n}. 
\]
Recalling (\ref{eq:tom}) and taking logarithms yields the right hand side of (\ref{eq:milana}).

For the remaining non trivial upper bound of (\ref{eq:malina}) put $M_{n}:=\scp{\zeta^{n}}\in{\mathbb{Z}}$ and consider
\[
 \Phi(M_{n}):= \prod_{j=1}^{k} \left(M_{n}-\zeta_{j}^{n}\right).
\]
The values $\phi(M_{n})$ are integers by Theorem \ref{elemsatz}, as it is again 
a symmetric polynomial in the variables $\zeta_{j}$ with integer coefficients.
 Moreover $\Phi(M_{n})\neq 0$ for all positive integers $n$.
Indeed, if otherwise $\zeta$ would be of the shape
$\zeta=\sqrt[m]{L}$ for some $L\in{\mathbb{Z}}$, but then its conjugates would be $\eta_{m}^{j}\zeta$ with
$\eta_{m}:=e^{\frac{2\pi i}{m}}$, contradicting the fact that the conjugates 
apart from $\zeta$ itself lie inside the unit circle. Clearly, powers of the other conjugates 
are smaller than $1$ in absolute value so they cannot equal $M_{n}$ either.

Thus we have
\[
 1\leq \prod_{j=1}^{k} \left\vert M_{n}-\zeta_{j}^{n}\right\vert= 
\left\vert M_{n}-\zeta^{n}\right\vert\prod_{j=2}^{k} \left\vert M_{n}-\zeta_{j}^{n}\right\vert,
\]
which leads to 
\begin{equation} \label{eq:franzr}
 \left\Vert \zeta^{n}\right\Vert = \vert M_{n}-\zeta^{n}\vert\geq \frac{1}{\prod_{j=2}^{k} \left\vert M_{n}-\zeta_{j}^{n}\right\vert}.
\end{equation}
However, $M_{n}=\zeta^{n}+o(1)$ and $\vert\zeta_{j}^{n}\vert=o(1)$ for $2\leq j\leq k$ as $n\to\infty$.
Using this in (\ref{eq:franzr}) gives 
\begin{equation} \label{eq:motte}
\Vert \zeta^{n}\Vert \geq (\zeta^{n}+o(1))^{-(k-1)}, \qquad n\to\infty.
\end{equation}
Now the fact $\overline{\sigma}(1,\zeta)$ cannot exceed $k-1$ follows easily.
If otherwise $\overline{\sigma}(1,\zeta)>k-1$, then for some arbitrarily large values of $n$ and some $\epsilon>0$ we would have
\[
\Vert \zeta^{n}\Vert \leq \zeta^{-n(1+\epsilon)(k-1)}=\left(\zeta^{1+\epsilon}\right)^{-n(k-1)}=\widehat{\zeta}^{-n(k-1)}
\]
with $\widehat{\zeta}:=\zeta^{1+\epsilon}>\zeta$, so $\lim_{n\to\infty} \widehat{\zeta}^{n}-\zeta^{n}=\infty$,
clearly contradicting (\ref{eq:motte}) for $n$ sufficiently large. 
\end{proof}

\begin{remark}
A generalization for the upper bound in (\ref{eq:malina})
for $\overline{\sigma}(\alpha,\zeta)$ for arbitrary algebraic 
real numbers $\alpha$ is given in Corollary \ref{kolmar}.
An interesting question is if in fact $\overline{\sigma}(\alpha,\zeta)=0$
for all Pisot numbers $\zeta$ and algebraic numbers $\alpha\notin{\mathbb{Q}(\zeta)}$ and
more general for any algebraic $\zeta$ and algebraic $\alpha\notin{\mathbb{Q}(\zeta)}$.
For transcendental $\alpha$ this is wrong, see Theorem \ref{kommtno}. 

For the quantity $\underline{\sigma}(\alpha,\zeta)$ with $\zeta$ a Pisot number and
real $\alpha\notin{\mathbb{Q}(\zeta)}$
we have $\underline{\sigma}(\alpha,\zeta)=0$ by Theorem \ref{gans}.
\end{remark}

Analyzing the proof of Theorem we obtain a Corollary that we will refer to in the sequel.

\begin{corollary}  \label{remark}
For a Pisot number $\zeta$ of degree $[\mathbb{Q}(\zeta):\mathbb{Q}]=k$
with conjugates $\zeta_{2},\zeta_{3},\ldots ,\zeta_{k}$ we have 
\[
\underline{\sigma}(1,\zeta)\leq -\frac{\log f}{\log \zeta}
\]
where $f:=\max_{2\leq j\leq k} \vert \zeta_{j}\vert<1$.
\end{corollary}

\begin{proof}
 The proof of the right hand side of (\ref{eq:milana}) effectively contained the proof of the claim. 
\end{proof}

The upper bound for $\overline{\sigma}(1,\zeta)$ seems to be not very strong. It seems unlikely 
that the case $\overline{\sigma}(1,\zeta)>\frac{1}{k-1}$, or more general 
\begin{equation} \label{eq:novak}
\overline{\sigma}(1,\zeta)>\underline{\sigma}(1,\zeta),
\end{equation}
does occur for any Pisot number $\zeta$. A stronger result should be true.

\begin{conjecture} \label{conjectur}
For any Pisot number $\zeta$ and any real $\alpha$ we have
$\underline{\sigma}(\alpha,\zeta)=\overline{\sigma}(\alpha,\zeta)$.
\end{conjecture}
 
We want to give a result that somehow quantifies this,
which connects the constant $\underline{\sigma}(1,\zeta)$ with the 
Diophantine problem (\ref{eq:vir}) for some $\boldsymbol{\zeta}$ arising from the
conjugates of $\zeta$. First we define

\begin{definition}
 Denote with $Liov$ the set of Liouville numbers, which is irrational real numbers $\zeta$ such
that $\lambda_{1}(\zeta)= \infty$.
\end{definition}

It is known that $Liov$ consists only of transcendental numbers, as irrational 
algebraic numbers $\zeta$
have $\lambda_{1}(\zeta)= 1$ by Roth`s Theorem, that can be found in \cite{5}, and (\ref{eq:hunt}).

Liov has Hausdorff dimension $0$, which is an easy consequence of the following Theorem by Bernik \cite{4}:


\begin{theorem}[Bernik]
 For a polynomial $P$ let $H(P)$ be the maximum absolute value of its 
coefficients and let
\[
A(w) := \left\{ \zeta\in{\mathbb{R}}: \vert P(\zeta)\vert < H(P)^{-w} 
\text{\quad for \quad infinitely \quad many \quad} P\in{\mathbb{Z}[X]}, \quad \text{deg}(P)\leq n\right\}
\]
Then the Hausdorff dimension of $A(w)$ equals $\frac{n+1}{w+1}$.
\end{theorem}

The special case $n=1$ proves the claim concerning the Hausdorff dimension. So the set $Liov$ is small in some sense,
which is of interest with respect to our following Theorem \ref{haaz}.
For its proof we need

\begin{theorem} [Smyth] \label{smyth}
Let $\zeta$ be a Pisot number and $\zeta_{1}=\zeta,\zeta_{2},\ldots,\zeta_{k}$ be its conjugates.
Then $\vert \zeta_{i}\vert =\vert \zeta_{j}\vert$ for $i\neq j$ implies $\overline{\zeta_{i}}=\zeta_{j}$.
\end{theorem}

See \cite{9} for a proof. We will also need a closely related deeper result.

\begin{theorem} [Mignotte] \label{naja}
For a Pisot number, there is no nontrivial multiplicative relation between its conjugates.  
\end{theorem}

See \cite{8} for a more precise definition and a proof. 

\begin{theorem} \label{haaz}
 Let $\zeta=\zeta_{1}$ be a Pisot number of degree $[\mathbb{Q}(\zeta):\mathbb{Q}]=k$.
 Further let the conjugates $\zeta_{1}=\zeta,\zeta_{2},\ldots,\zeta_{k}$
be labeled by decreasing absolute values, such that in particular
$\vert \zeta_{2}\vert= \max_{2\leq j\leq k} \vert \zeta_{j}\vert$.
Put $\zeta_{2}=R_{2}e^{i\psi_{2}}$ with $0\leq \psi_{2}<2\pi$.

Suppose $\overline{\sigma}(1,\zeta)> \underline{\sigma}(1,\zeta)$, 
which by Theorem \ref{zauber} in particular holds if
$\overline{\sigma}(1,\zeta)>\frac{1}{k-1}$. Then the following holds

\[ 
\frac{\psi_{2}}{2\pi}\in{Liov}.	 
\]
\end{theorem}
	 
\begin{proof}
First note that by Smyth's Theorem \ref{smyth} the conjugate $\zeta_{2}$ is 
determined by the absolute value property up to complex conjugation.
Define $r_{j}$ by $\vert \zeta_{j}\vert= r_{j}\vert \zeta_{2}\vert=r_{j}R_{2}$ for $3\leq j\leq k$.
Put $\zeta_{3}=\overline{\zeta_{2}}$ in case of non-real $\zeta_{2}$.
Then, clearly $0<r_{k}\leq r_{k-1}\ldots\leq r_{3}<1$ in case of real $\zeta_{2}$ as well as
 $0<r_{k}\leq r_{k-1}\leq \ldots \leq r_{4}<r_{3}=1$ in case of non-real $\zeta_{2}$
 by Smyth's Theorem \ref{smyth}.

In view of this, if $\zeta_{2}$ is real, i.e. $\psi_{2}=0$, for $n$ sufficiently large
\begin{equation} \label{eq:reell}
\Vert \zeta^{n}\Vert = \left\vert \sum_{j=2}^{k} \zeta_{j}^{n}\right\vert 
\geq \left\vert\zeta_{2}\right\vert^{n}\left(1-\sum_{j=3}^{k} r_{j}^{n}\right)
\geq \frac{1}{2}\cdot \left\vert\zeta_{2}\right\vert^{n}, \qquad n\geq n_{0}.
\end{equation}
Hence and by Corollary \ref{remark} this case clearly yields
$\overline{\sigma}(1,\zeta)= \underline{\sigma}(1,\zeta)\leq -\frac{\log\vert\zeta_{2}\vert}{\log \zeta}\leq \frac{1}{k-1}$.

If otherwise $\psi_{2}\neq 0$, we similarly have
\begin{equation} \label{komplex}
 \Vert \zeta^{n}\Vert = \left\vert \sum_{j=2}^{k} \zeta_{j}^{n}\right\vert 
\geq \left\vert\zeta_{2}^{n}\left(2 \vert\rm{cos}(n\psi_{2})\vert-\sum_{j=4}^{k} r_{j}^{n}\right)\right\vert
\geq \left\vert\zeta_{2}\right\vert^{n}\left(2 \rm{cos}(n\psi_{2})-(k-3)r_{4}^{n}\right), \qquad n\geq n_{0}.
\end{equation} 
Thus by Corollary \ref{remark} and noting 
\begin{equation} \label{eq:oft}
0<r_{4}<\frac{1+r_{4}}{2}<1,
\end{equation}
if $\overline{\sigma}(1,\zeta)> \underline{\sigma}(1,\zeta)= -\frac{\log\vert\zeta_{2}\vert}{\log \zeta}$,
 we must have a sequence of values $n$ such that 
\begin{equation} \label{eq:diese}
 \rm{cos}(n\psi_{2})= o\left(\left(\frac{1+r_{4}}{2}\right)^{n}\right), \qquad n\to\infty.
\end{equation}
To avoid subindices we only consider this sequence of values $n$ in the sequel.
Using twice the addition Theorem 
$\rm{cos}(2a)=\rm{cos}^{2}(a)-\sin^{2}(a)=2\rm{cos}^{2}(a)-1$ for cosine we infer
 $\rm{cos}(4n\psi_{2})=8\rm{cos}^{4}(n\psi_{2})-4\rm{cos}^{2}(n\psi_{2})+1=1+o\left(\left(\frac{1+r_{4}}{2}\right)^{2n}\right)$
as $n\to\infty$.
Hence 
\begin{equation} \label{eq:mueh}
\sin(4n\psi_{2})=\sqrt{1-\rm{cos}^{2}(4n\psi_{2})}=o\left(\left(\frac{1+r_{4}}{2}\right)^{n}\right), \qquad n\to\infty.
\end{equation}
But by (\ref{eq:oft}) the right hand side of (\ref{eq:mueh}) tends to $0$.
Clearly, this implies 
\begin{equation} \label{eq:selten}
4n\psi_{2}=2\pi m_{n}+o(1)
\end{equation}
as $n\to\infty$ for an integer sequence $(m_{n})_{n\geq 1}$.
 Using $\frac{1}{2}\vert x\vert \leq \vert \rm{sin}(x+2m\pi)\vert$ for an integer $m$ and $x\in{[-0.1,0.1]}$,
(\ref{eq:mueh}) yields 
\[
4n\psi_{2}=2m_{n}\pi+o\left(\left(\frac{1+r_{4}}{2}\right)^{n}\right)
\]
too for the sequence in (\ref{eq:diese}) and the corresponding integer sequence $(m_{n})_{n\geq 1}$
from (\ref{eq:selten}).
But again by (\ref{eq:oft}), this in particular implies $4n\psi_{2}=2m_{n}\pi+o(n^{-\nu})$ or equivalently
$(4n)\cdot \frac{\psi_{2}}{2\pi}-m_{n}=o(n^{-\nu})$ for any $\nu>0$ and $n$ sufficiently large.
By definition this implies $\frac{\psi_{2}}{2\pi}$ is either nonzero rational (the case
of a real numbers $\zeta_{2}$ was treated in (\ref{eq:reell})) or a Liouville number. 
To see that it cannot be rational we use Mingnotte's Theorem \ref{naja}. If $\frac{\psi_{2}}{2\pi}$
is nonzero rational, then $\zeta_{2}$ would be a real multiple of a root of unity.
In this case for some positive integer $L$ we would have 
$\zeta_{2}^{2L}= \vert \zeta_{2}\vert^{2L}= \zeta_{2}^{L}\overline{\zeta_{2}}^{L}$, 
so $\zeta_{2}^{L}=\overline{\zeta_{2}}^{L}$,
but this clearly is a nontrivial multiplicative relation between the roots $\zeta_{2},\zeta_{3}=\overline{\zeta_{2}}$
contradicting Mignotte's Theorem.   
\end{proof}
 	 
Theorem \ref{haaz} in fact even shows the stronger condition 
\[
 \left\Vert n\frac{\phi_{2}}{2\pi}\right\Vert \leq e^{-\nu n}
\]
has infinitely many integral solutions $n$ for some $\nu>0$ in 
case of $\overline{\sigma}(1,\zeta)> \underline{\sigma}(1,\zeta)$.
However, as algebraic numbers are only countable and angles $\frac{\psi_{j}}{2\pi}$ in Theorem \ref{haaz} 
 are typically expected to be transcendental, the pathological case cannot be excluded easily.
However, we give another Corollary affirming it should not happen. 

\begin{corollary} \label{zzz}
If for any real $\theta$ in the splitting field of a Pisot polynomial the value $\frac{\rm{arctan}(\theta)}{\pi}$ is  
not a Liouville number, then $\underline{\sigma}(1,\zeta)=\overline{\sigma}(1,\zeta)$
for any Pisot number $\zeta$.
In particular, if for any algebraic number $\theta$ the value $\frac{\rm{arctan}(\theta)}{\pi}$ is either rational 
or irrational and no Liouville number, then $\underline{\sigma}(1,\zeta)=\overline{\sigma}(1,\zeta)$
for any Pisot number $\zeta$.
\end{corollary}

\begin{proof}
  For any algebraic number $\gamma=re^{i\psi}$, since
the conjugate is algebraic in the same splitting field so are $\rm{Re}(\gamma)=\frac{\gamma+\overline{\gamma}}{2},
\rm{Im}(\gamma)=\frac{\gamma-\overline{\gamma}}{2i}$ and hence their quotient
 $\frac{\gamma-\overline{\gamma}}{i(\gamma+\overline{\gamma})}=:\theta=\rm{tan}(\psi)\in{\mathbb{R}}$ too.
So we may apply Theorem \ref{haaz} using its notation with $\gamma=\zeta_{2}$
(or equivalently $\theta=\rm{tan}(\psi_{2})$ or $\rm{arctan}(\theta)=\psi_{2}$).  
\end{proof}

\begin{remark}
The condition in Corollary \ref{zzz} is equivalent to the condition that 
 $w_{2}(\pi,\tan(\psi_{2}))=\infty$ for the two-dimensional simultaneous approximation constant $w_{2}$
in (\ref{eq:rotzz}). It is worth noting by Khinchin's transference principle Theorem \ref{khinchin} this
implies $\lambda_{2}(\pi,\rm{tan}(\psi_{2}))\geq 1$.
\end{remark}

The next Proposition shows that indeed for $k=2$, we cannot have (\ref{eq:novak}). 

\begin{theorem} \label{kaffee}
 Let $\zeta$ be a Pisot number of degree $[\mathbb{Q}(\zeta):\mathbb{Q}]=k=2$.
 Then $\sigma_{n}(1,\zeta)$ is eventually constant. In particular
 $\underline{\sigma}(1,\zeta)=\overline{\sigma}(1,\zeta)$.

Furthermore the set
\begin{equation} \label{eq:hipp}
 \left\{ t\in{\mathbb{R}}: \exists \zeta \quad \text{Pisot number of degree k=2 such that} \quad
 \underline{\sigma}(1,\zeta)= \overline{\sigma}(1,\zeta)=t\right\}
\end{equation}
is dense in $(0,1]$. Moreover 
\begin{equation}  \label{eq:neu}
\underline{\sigma}(1,\zeta)=\overline{\sigma}(1,\zeta)=1
\end{equation}
if and only if $\zeta$ is an algebraic unit.
\end{theorem}

\begin{proof}
In the quadradic case the only conjugate $\zeta_{1}$ of $\zeta$ is real, so for $n\geq n_{0}$ 
large enough that $\vert \zeta_{1}\vert^{n}<\frac{1}{2}$ we have 
$\Vert\zeta^{n}\Vert= \vert\zeta_{1}\vert^{n}$, thus by definition
\begin{equation} \label{eq:foma}
\sigma_{n}(1,\zeta)=-\frac{\log \vert\zeta_{1}\vert}{\log \zeta}\in{(0,1]}, \qquad n\geq n_{0}
\end{equation}
For the second statement first consider only $\zeta$ of the shape $\zeta=N+\sqrt{d}$ and
observe that for positive integers $N,d$ the conditions
$(N-1)^{2}+1 \leq d \leq (N+1)^{2}-1, d\neq N^{2}$, are easily seen to be necessary and sufficient for
 such $\zeta$ to be a Pisot number of degree $2$. We restrict to $(N-1)^{2}+1 \leq d \leq N^{2}-1$.

Let $N$ tend to infinity and according to (\ref{eq:foma}) consider the values
\[
 \chi(N,d):=\frac{\log \vert N-\sqrt{d}\vert}{\log N+\sqrt{d}}, \qquad (N-1)^{2}-1\leq d\leq N^{2}-1.
\]
For $d=N^{2}-1$ we have $(N+\sqrt{d})=(N-\sqrt{d})^{-1}$, so in this case $\zeta=N+\sqrt{d}$ is
an algebraic unit and $\underline{\sigma}(1,\zeta)=\overline{\sigma}(1,\zeta)=1$.
Similarly for $d=N^{2}+1$. Conversely, let an arbitrary quadratic irrational $\zeta=N+M\sqrt{d}$
for some $M,N,d\in{\mathbb{Z}}$ be a unit in $\mathbb{Q}(\sqrt{d})$. Units are 
known to be precisely the elements with norm
$\rm{N}_{\mathbb{Q}(\sqrt{d})/\mathbb{Q}}(\zeta)=\zeta\zeta_{1}\in{\{-1,1\}}$, where $\zeta_{1}$ is the conjugate.
But the latter is just a reformulation of (\ref{eq:neu}).
 
For the other interval end $d=(N-1)^{2}+1$, the numerator is easily seen to tend
to $0$ whereas the denominator tends to infinity, so with $N\to \infty$ the quotient 
tends to $0$. To sum up for any fixed $N$, and with $d$ increasing
 in the interval $(N-1)^{2}-1\leq d\leq N^{2}-1$, the values $\chi(N,d)$ decrease from $1$ to $\epsilon(N)$
which tends to $0$ as $N\to\infty$. 
Consequently, to obtain that (\ref{eq:hipp}) is dense it remains to prove that with
\[
 \Phi(x,y):=\chi(x,y)-\chi(x,y+1)=\frac{\log (x-\sqrt{y})}{\log (x+\sqrt{y})}- \frac{\log (x-\sqrt{y+1})}{\log (x+\sqrt{y+1})},
 \qquad (x-1)^{2}+1\leq y\leq x^{2}-2
\]  
the value $\max_{y}\Phi(x,y)$ tends to $0$ as $x\to\infty$, where the maximum is
taken over $y$ in the given interval. This is standard analysis. 

Using $\Phi(x,y)\leq \vert \int_{y,y+1} \chi_{y}(x,t)dt\vert \leq \max \vert \chi_{y}(x,y)\vert$
with the restriction $(x-1)^{2}+1\leq y\leq x^{2}-2$ some computation shows
\begin{eqnarray*}
 \Phi(x,y)&\leq& \left\vert \frac{-\frac{1}{2\sqrt{y}(x-2\sqrt{y})}\log(x+\sqrt{y})-\frac{1}{2\sqrt{y}(x+2\sqrt{y})}\log(x-\sqrt{y})}{\log(x+\sqrt{y})^{2}}  \right\vert   \\
  &\leq& \left\vert\frac{\frac{1}{x-\sqrt{y}}-\frac{1}{x+\sqrt{y}}}{2\sqrt{y}\log(x+\sqrt{y})}\right\vert    \\
	&=& \frac{\sqrt{y}}{\log(x+\sqrt{y})\sqrt{y}(x^{2}-y)}                                   \\
	&=& \frac{1}{\log(x+\sqrt{y})(x^{2}-y)},
\end{eqnarray*}
and by $x^{2}-y\geq 1$ this indeed tends to $0$.

\end{proof}	 

We are interested in higher dimensional generalizations of Theorem \ref{kaffee}.
The point if $\underline{\sigma}(1,\zeta)=\overline{\sigma}(1,\zeta)$ was discussed
in Theorem \ref{haaz}.
Counterexamples to some of the facts are summed up in Theorem \ref{proko}.
The claims there are quite natural and suggestive but the proofs are a little technical.
We will make use of Rouche's Theorem

\begin{theorem}[Rouche]  \label{rouchee}
 Let $f,g$ be holomorphic functions defined on a simply connected open subset $K$ of the complex plane.
 If $\vert g\vert \leq \vert f\vert$ on the boundary of $K$, then $f$ and $f+g$ have the same
number of zeros in $K$.
\end{theorem}

See \cite{12} for a proof.

We will make use of a special class of Pisot polynomials in the sequel.

\begin{proposition} \label{pisotpol}
Let $k\geq 2$ and $M,N$ be integers satisfying $M\geq 3$ and $1\leq N\leq M-2$.
Then the polynomial
\[
 Q_{k,M,N}(X):= X^{k}-MX^{k-1}+N
\] 
is a Pisot polynomial.
\end{proposition}

\begin{proof}
We apply Rouche's Theorem with $f(z):= Mz^{k-1}-N$,
$g(z):=z^{k}$ and the closed unit circle as $K$ to prove that
$Q_{k,M,N}$ is a Pisot polynomial. 
Indeed, on the one hand the $k-1$ zeros of $Mz^{k-1}-N$ are clearly inside the unit circle
as they all have absolute value $\sqrt[k]{\frac{N}{M}}<1$ (and differ only by 
multiplication with a $k$-th root of unity), by the conditions on $M,N$
for $\vert z\vert=1$ we have
 $\vert g(z)\vert=\vert z^{k}\vert=1<M-N<\vert Mz^{k-1}-N\vert=\vert f(z)\vert$, so
by Rouchees Theorem $Q_{k,M,N}(X)$ has $k-1$ roots inside the unit circle too.
On the other hand, it is easily seen that the remaining root is real 
and greater $1$ by the intermediate value Theorem, as $Q_{k,M,N}(1)= 1-M+N<0$ and
$\lim_{x\to\infty} Q_{k,M,N}(x)=+\infty$ because the leading coefficient of $Q_{k,M,N}$
is positive.
\end{proof}

\begin{remark}
The proof essentially shows that more general every polynomial $P(X)=X^{k}+a_{k-1}X^{k-1}+\cdots +a_{0}\in{\mathbb{Z}[X]}$
with $\vert a_{k-1}\vert \geq \vert a_{k-2}\vert+\vert a_{k-3}\vert+\cdots +\vert a_{0}\vert+2$
is a Pisot polynomial if its unique real root of absolute value greater than $1$
is positive (it is real because complex roots have a conjugate of same absolute value).
It is not hard to see by Vieta Theorem \ref{vieta} that the positivity is equivalent to $a_{k-1}<0$.
Note that nevertheless in the other case of a negative root the basic approximation properties coincide. 
\end{remark}
								
\begin{theorem} \label{proko}
Let $\zeta$ be a Pisot number of degree $[\mathbb{Q}(\zeta):\mathbb{Q}]=k\geq 3$.
Then $\sigma_{n}(1,\zeta)$ is not ultimately constant. Moreover, there are
algebraic Pisot units $\zeta$ of any given degree $k$ such that
\begin{equation}  \label{eq:neuere}
\underline{\sigma}(1,\zeta)<\frac{1}{k-1}.
\end{equation}
\end{theorem} 

\begin{proof}
Assume $\sigma_{n}(\zeta)$ would be ultimately constant.
Recall (\ref{eq:hansbirschl}) and without loss of generality
let $\vert\zeta_{2}\vert=\max_{2\leq j\leq k} \vert \zeta_{j}\vert$.
For $k\geq 3$ this gives
\begin{eqnarray*}
 \sigma_{n}(1,\zeta)&=& \frac{\log \Vert \zeta^{n}\Vert}{n\cdot \log \zeta}
=\frac{\log\sum_{j=2}^{k}\zeta_{j}^{n}}{n\cdot \log \zeta}
=\frac{\log\left(\zeta_{2}^{n}\cdot(1+\sum_{j=3}^{k}(\frac{\zeta_{j}}{\zeta_{2}})^{n})\right)}{n\cdot \log \zeta}  \\
&=& \frac{\log \zeta_{2}}{\log \zeta}+ \frac{\log\left(1+\sum_{j=3}^{k}(\frac{\zeta_{j}}{\zeta_{2}})^{n}\right)}{n\cdot \log(\zeta)}.
\end{eqnarray*}
Now if this is eventually constant so is the second term, applying exponential map yields
\begin{equation}  \label{eq:wiener}
 \zeta^{C\cdot n}= 1+\sum_{j=3}^{k}\left(\frac{\zeta_{j}}{\zeta_{2}}\right)^{n}, \qquad C\in{\mathbb{R}},n\geq n_{0}.
\end{equation}
By our assumption the right hand side in (\ref{eq:wiener}) is bounded by $k-1$, so $C\leq 0$.
Now using the same method as for the proof of the right hand side in (\ref{eq:milana})
for $\frac{\zeta_{j}}{\zeta_{n}}$ we see that there are arbitrarily large $n$ such
that every element of the sum on the right hand side is positive, so the sum is (strictly) greater than $1$.
Thus also $C\geq 0$ and hence $C=0$, hence the right hand side in (\ref{eq:wiener}) equals $1$
for all $n\geq n_{0}$. But in the argument for $C\geq 0$ we saw that this is impossible, a contradiction.

For the second claim recall (\ref{eq:tom}) in context of the proof of the right hand side
of (\ref{eq:milana}), and Corollary \ref{remark}. In view to this it suffices 
 for fixed degree $k\geq 3$ to find
a Pisot unit whose conjugates do not all lie on the circle with
origin $0$ and radius $\zeta^{-\frac{1}{k-1}}$ in the complex plane.
	
We prove that for any fixed $k$, Pisot numbers associated to the set of polynomials $Q_{k,M,1}$ arising
from Proposition \ref{pisotpol} and putting $N=1$ have this property.

Firstly, by Proposition \ref{pisotpol} these are Pisot polynomials, and as the constant term equals $1$ the associated
Pisot number is a unit by Vieta Theorem \ref{vieta}.
To prove our claim, we show that there is a real root $\zeta_{2,M}\in{(0,1)}$ which has strictly larger
absolute value than any other root inside the unit circle (i.e. apart from the Pisot number
associated to $Q_{k,M,1}$).
The existence of a real root $\zeta_{2,M}\in{(0,1)}$ of $Q_{k,M,1}$
is immediate due to $Q_{k,M,1}(0)=1>0, Q_{k,M,1}(1)=2-M<0$ 
by intermediate value Theorem. The uniqueness is easily inferred by looking at the derivative
$Q_{k,M,1}^{\prime}(x)=kx^{k-1}-M(k-1)x^{k-2}$ which obviously has unique positive root $x=\frac{M(k-1)}{k}$,
consequently apart from the Pisot root there can only be one more positive real root. 

Note that $Q_{k,M,1}(z)=0$ is equivalent to $Mz^{k-1}=z^{k}+1$,
so in particular $M\vert z\vert^{k-1}= \vert z^{k}+1\vert$. Above we saw
$Q_{k,M,1}(t)<0$ for $t\in{(\zeta_{2,M},1)}$. Applying this to $t:=\vert \zeta_{3}\vert$ with
another root $\zeta_{3,M}\neq \zeta_{2,M}$ of $Q_{k,M,1}$ with $\vert \zeta_{3,M}\vert\in{[\zeta_{2,M},1)}$, 
we must have $\vert \zeta_{3,M}\vert=\zeta_{2,M}$. More precisely, since
$\vert z^{k}+1\vert\leq \vert z\vert^{k}+1$ with equality if and only if $z^{k}$ is real,
$\zeta_{3,M}$ must be of the shape
$\zeta_{3,M}=\zeta_{2,M}e^{\frac{2\pi ij}{k}}$, $j\in{\{1,2,\ldots ,k-1\}}$. However, for these
values of $\zeta_{3,M}$ we have $M\zeta_{3,M}^{k-1}\notin{\mathbb{R}}$ as $(k-1,k)$ are
relatively prime, so $Q_{k,M,1}(\zeta_{3,M})\notin{\mathbb{R}}$,
so $Q_{k,M,1}(\zeta_{3,M})\neq 0$ in particular, contradicting the existence
of a root of $Q_{k,M,1}$ with absolute value in $[\zeta_{2,M},1)$. 
\end{proof}

\begin{remark}
 Applying the nontrivial result of C. Smyth \cite{9} we already referred to
in Theorem \ref{haaz} would have simplified the proof of the second part of Theorem \ref{proko}.
However, we wanted to present a more elementary proof.
More general, Smyth's Theorem implies

\begin{corollary}  \label{bintroll}
 Every Pisot number $\zeta$ of degree $[\mathbb{Q}(\zeta):\mathbb{Q}]\geq 4$ satisfies
$\underline{\sigma}(1,\zeta)< \frac{1}{k-1}$.
\end{corollary}

\begin{proof}
By Smyth's Theorem \ref{smyth}, there can be at most $2$ roots of every fixed absolute value $f$.
So, for $k\geq 4$ there must be two conjugates of $\zeta$ inside the unit circle with different absolute values.
Looking at the constant coefficient of the Pisot polynomial associated to $\zeta$
it follows by Vieta Theorem \ref{vieta} that the maximum absolute value of a root inside the circle is strictly larger 
than $\zeta^{-\frac{1}{k-1}}$. Thus by Corollary \ref{remark} we have
$\underline{\sigma}(1,\zeta)< \frac{1}{k-1}$ in this case. 
\end{proof}
 
\end{remark}

We now inspect the problem if (\ref{eq:hipp}) with $2$ changed to $k$ is dense is true for $k\geq 3$.
Using the polynomials from Proposition \ref{pisotpol} we can give an affirmative answer.

\begin{theorem}
 For any $k\geq 2$, (\ref{eq:hipp}) is dense in $\left(0,\frac{1}{k-1}\right]$.
\end{theorem}

\begin{proof}
Fix $\epsilon\in{[0,1)}$ and 
consider the polynomials $Q_{k,M,N}$ from Proposition \ref{pisotpol}
with $N=\left\lfloor M^{-\epsilon}\right\rfloor$ and let $M\to\infty$.
Clearly, for $M$ sufficiently large the conditions on $M,N$ are satisfied.
Let $z$ be any root unequal the Pisot number $\zeta_{k,M,N}$ associated to $Q_{k,M,N}$.
Since $Q_{k,M,N}$ is Pisot polynomial it has $\vert z\vert \leq 1$.
The equation $Q_{k,M,N}(z)=0$ implies $z^{k}+\left\lfloor M^{-\epsilon}\right\rfloor= Mz^{k-1}$ and further
\[
 M^{-\epsilon}-2 \leq M\vert z\vert^{k-1} \leq M^{-\epsilon}+2,
\]
thus for $M\to\infty$ for every root of $Q_{k,M,\left\lfloor M^{-\epsilon}\right\rfloor}$
we have the asymptotic property
\begin{equation} \label{eq:basned}
\vert z\vert \thicksim M^{\frac{-1-\epsilon}{k-1}}.
\end{equation}
However, by $Q_{k,M,N}(M-1)= -(M-1)^{k-1}+N<0$ and $Q_{k,M,N}(M)=1>0$
by intermediate value Theorem 
we have $M-1< \zeta_{k,M,N}< M$ for the Pisot number $\zeta_{k,M,N}$ associated to any polynomial $Q_{k,M,N}$.
So by (\ref{eq:basned}) for $M\to\infty$ we also have the asymptotic property
\[
\vert z\vert \thicksim \zeta_{k,M,N}^{\frac{-1-\epsilon}{k-1}}.
\]
Varying $\epsilon$
proves in view to Corollary \ref{remark} that $\underline{\sigma}(1,\zeta)$ is dense.
However, an obvious generalization of the proof of Theorem \ref{proko} for arbitrary $N$ shows
that there is a real root of $Q_{k,M,N}$ 
with the unique maximum absolute value among all the roots of $Q_{k,M,N}$ apart from the Pisot root $\zeta_{k,M,N}$.
So with the same argument as in (\ref{eq:reell}) in the proof of Theorem \ref{haaz},
$\underline{\sigma}(1,\zeta_{k,M,N})=\overline{\sigma}(1,\zeta_{k,M,N})$ for any $k,M,N$
satisfying the conditions of Theorem \ref{pisotpol}.                                   
\end{proof}

Another easy confirmative result relating to Theorem \ref{kaffee} is

\begin{proposition} \label{keintroll}
If for a Pisot number $\zeta$ of degree $[\mathbb{Q}(\zeta):\mathbb{Q}]=k$ we have
 $\underline{\sigma}(1,\zeta)=\frac{1}{k-1}$, then $\zeta$ is a Pisot unit. 
\end{proposition}

\begin{proof}
 If $\zeta$ is not a unit, then the constant coefficient of its minimal polynomial $P(X)$ has
absolute value at least $2$, so the largest conjugate of $\zeta$ has absolute value at least
$\frac{\log 2}{\log \zeta^{k-1}}> \frac{1}{(k-1)\log \zeta}$ by Vieta Theorem \ref{vieta}. 
The proof of the right hand side of (\ref{eq:milana}), i.e. Corollary \ref{remark}, 
finishes the proof. 
\end{proof}

\begin{remark}
Only in case of $k=3$ Proposition \ref{keintroll} is 
not already implied by Theorem \ref{kaffee} or Corollary \ref{bintroll},
which shows the set is in fact empty for $k\geq 4$.
\end{remark}

\subsection{Case of algebraic numbers $\alpha,\zeta>1$ and $\zeta$ not a Pisot number}  \label{iii}

Observe that Theorem \ref{gans} applies in this case, so $\underline{\sigma}(\alpha,\zeta)=0$.
First we point out that converse to the Conjecture \ref{conjectur} from section \ref{hhh} that this should never 
occur for Pisot numbers, there exist algebraic numbers $\alpha,\zeta>1$
with $\underline{\sigma}(\alpha,\zeta)<\overline{\sigma}(\alpha,\zeta)$.

\begin{proposition}
There exist algebraic numbers $\zeta$ with $\underline{\sigma}(1,\zeta)<\overline{\sigma}(1,\zeta)$. 
\end{proposition}

\begin{proof}
Take $\nu$ any Pisot number and $k$ a positive integer such that $\zeta:=\sqrt[k]{\nu}$
is no Pisot number. Any sufficiently large $k$ satisfies this property since there exists a smallest
Pisot number which happens to be $\rho\approx 1.3247$ the root of $X^{3}-X-1$ see Theorem 7.2.1 on
page $133$ in \cite{101}, but otherwise $\sqrt[k]{\nu}$ converges to $1$.
Since $\zeta$ is no Pisot number we have $\underline{\sigma}(1,\zeta)=0$ by Pisot Theorem \ref{pisot}.
On the other hand 
$\overline{\sigma}(1,\zeta)\geq \overline{\sigma}(1,\nu)\geq \underline{\sigma}(1,\nu)>0$ 
by (\ref{eq:nanoleuchte}) and Theorem \ref{zauber}.  
\end{proof}

We want to give estimates for $\overline{\sigma}(\alpha,\zeta)$. Before we proof our main
result we give an easy Proposition we will later need.

\begin{proposition}  \label{tuertor}
 For real numbers $\alpha\neq 0,\zeta\neq 0$ if $\alpha\zeta^{n}$ is 
an integer for infinitely many positive integers $n$,
 then $\zeta=\sqrt[L]{M}$ for positive integers $L,M$ and
$\alpha=\frac{A}{B}\zeta^{-g}\in{\mathbb{Q}(\zeta)}$ for $A,B$ integers and $g$ a non-negative integer. 
In particular $\alpha\in{\mathbb{Q}(\zeta)}$.
\end{proposition}

\begin{proof}
First note that without loss of generality we may assume $\alpha,\zeta$ to be positive
and then, clearly $\lambda>1$ is a necessary condition (so we are in the interesting case). 

Let $n_{2}>n_{1}$ be two such integers. Then
 we can write $\alpha\zeta^{n_{1}}=M_{1}, \alpha\zeta^{n_{2}}=M_{2}$
with integers $M_{2}>M_{1}$.
 Then building quotients we get 
\begin{equation} \label{eq:tal}
\zeta=\chi:=\sqrt[n_{2}-n_{1}]{\frac{M_{2}}{M_{1}}}.
\end{equation}
Let $\frac{m_{2}}{m_{1}}=\frac{M_{2}}{M_{1}}$ with $(m_{1},m_{2})=1$ relatively prime.
If $m_{1}=1$ put $L:=n_{2}-n_{1},M:=m_{2}$ and we are done. So suppose $m_{1}\geq 2$.

If $\alpha \zeta^{n}$ is an integer for infinitely many $n$, then by pigeon hole 
principle there is a residue class $f$ mod $n_{2}-n_{1}$ with infinitely many too, i.e.
\[
 \alpha \zeta^{N(n_{2}-n_{1})+f}\in{\mathbb{Z}}
\]
has infinitely many positive integer solutions $N$. But by (\ref{eq:tal}) the same applies to
\[
 \alpha \zeta^{f}\cdot \left(\frac{m_{2}}{m_{1}}\right)^{N}.
\]
But by $(m_{1},m_{2})=1$ this means $\alpha \zeta^{f}$ is an integer that
is divided by arbitrarily large powers of $m_{1}$, a contradiction to $m_{1}\geq 2$.

The claim concerning $\alpha$ follows by the fact that $\zeta^{n}$ is of the form
$BM^{\frac{g}{L}}$ with an integer $B$ and $g\in{\{0,1,\ldots, L-1\}}$, so for this
be an integer $A$ requires $\alpha$ to be of the form 
$\alpha=\frac{A}{B} M^{\frac{-g}{L}}=\frac{A}{B}\zeta^{-g}\in{\mathbb{Q}(\zeta)}$. 
\end{proof}

Another very easy Proposition to simplify the proof of Proposition \ref{elendiglich} later.

\begin{proposition}  \label{grillip}
If $\zeta=\sqrt[L]{M}$ for positive integers $L,M$ and $L$ chosen
minimal with this property, then for $t$ an integer $\zeta^{t}$ is rational if and only if
 $L\vert t$.
\end{proposition}

\begin{proof}
 Clearly if $LK=t$ with $K\in{\mathbb{Z}}$, then $\zeta^{t}=M^{K}$ is rational.
 Conversely, if $\zeta^{q}\in{\mathbb{Q}}$ for some integer $L\nmid q$, then since $\zeta^{L}=M$
also $q+pL$ has the same property for any $p\in{\mathbb{Z}}$, so 
there is an integer $1\leq s\leq L-1$ with this property $\zeta^{s}\in{\mathbb{Q}}$.
We may assume $s$ to be minimal with this property.
 Since $\zeta$ is irrational (otherwise $\zeta\in{\mathbb{Z}}$ and $L=1$)
 in fact $2\leq s\leq L-1$.
  Write $L=vs+u$ with $v,0\leq u\leq s-1$ integers.
By the assumed minimality of $L$ such that $\zeta^{L}$ is an integer
we must have $u\neq 0$, because otherwise if $\zeta^{s}\in{\mathbb{Q}\setminus \mathbb{Z}}$
clearly $\zeta^{L}=\zeta^{vs}$ is no integer, contradiction. 
But $\zeta^{L}\in{\mathbb{Z}}$ and $\zeta^{s}\in{\mathbb{Q}}$
 implies $\zeta^{L-vs}=\zeta^{u}\in{\mathbb{Q}}$, by $1\leq u\leq s-1$ 
a contradiction to the minimality of $s$. 
\end{proof}

Proposition \ref{tuertor} suggests that in order to investigate
 the quantity $\overline{\sigma}(\alpha,\zeta)$
we need to treat the case $\zeta=\sqrt[L]{M}$ separately.
We first define

\begin{definition}
 For real numbers $\alpha\neq 0,\zeta>1$ define the restricted approximation 
constant $\widehat{\overline{\sigma}}(\alpha,\zeta)$ as in (\ref{eq:mozart}) but excluded
the values $n$ for which $\sigma_{n}(\alpha,\zeta)=\infty$, i.e. $\alpha\zeta^{n}$ 
is an integer.
\end{definition}

\begin{proposition}  \label{elendiglich}
Let $\zeta=\sqrt[L]{M}$ for positive integers $M,L$ and $L$ chosen minimal with this property.
Further let $\mathscr{P}(N)$ denote the set of prime divisors of an integer $N$.
 Then for $\alpha=\frac{A}{B}\zeta^{-g}$ for non-negative integers $A,B,g$, $(A,B)=1$ 
with $\mathscr{P}(B)\subset {\mathscr{P}(M)}$ we have
$\underline{\sigma}(\alpha,\zeta)=\infty$, and for any other real algebraic $\alpha\neq 0$
we have $\underline{\sigma}(\alpha,\zeta)\leq 1$.
In any case $\widehat{\overline{\sigma}}(\alpha,\zeta)\leq 1$.
\end{proposition}  

\begin{proof}
For $\alpha=\frac{A}{B}\zeta^{-g}$ and $\mathscr{P}(B)\subset {\mathscr{P}(M)}$
we have $\alpha \zeta^{LN+g}= \frac{A}{B}M^{N}\in{\mathbb{Z}}$ for any 
positive integer $N\geq N_{0}$ with $N_{0}$ sufficiently large that $B\vert M^{N_{0}}$, so 
$\overline{\sigma}(\alpha,\zeta)=\infty$ in this case.

If otherwise $\alpha$ is not of this shape, we first show
that $\alpha\zeta^{n}$ is no integer for all sufficiently large $n$.

Assume the opposite.
By Proposition \ref{tuertor} then we must have $\alpha=\zeta^{-g}\frac{A}{B}$ for some $g\geq 0$.
In this case if we write $n=LN+r$ with $N,0\leq r \leq L-1$ integers.
Consequently $\alpha \zeta^{n}= \frac{A}{B}M^{N}\zeta^{r-g}\in{\mathbb{Z}}$. 
But $ \frac{A}{B}M^{N}\in{\mathbb{Q}}$, so 
\begin{equation}  \label{eq:panis}
\zeta^{r-g}\in{\mathbb{Q}}
\end{equation}
too.
By the assumption $\mathscr{P}(B)\nsubseteq {\mathscr{P}(M)}$ obviously
$\frac{A}{B}M^{t}$ is not an integer for any integer $t$, so $L\nmid r-g$. 
However, for $L\nmid r-g$ we have $\zeta^{r-g}\notin{\mathbb{Q}}$ by Proposition \ref{grillip},
a contradiction to (\ref{eq:panis}). So indeed there are only finitely many $n$ 
with $\alpha\zeta^{n}\in{\mathbb{Z}}$. 

Write $n=LN+r$ with $N,0\leq r \leq L-1$ integers. We have
\begin{equation}  \label{eq:huhn}
 \left\Vert \alpha\zeta^{n}\right\Vert = \left\Vert M^{N}(\alpha \zeta^{r})\right\Vert. 
\end{equation}
Observe all the $L$ values $\alpha \zeta^{r}$ for $0\leq r\leq L-1$ 
are algebraic and we showed $\alpha\zeta^{n}$ are no integers. 
If $\alpha \zeta^{r}$ is rational, i.e. $\alpha\zeta^{r}=\frac{p_{r}}{q_{r}}$, 
then the right hand side of (\ref{eq:huhn}) is bounded below by $\frac{1}{q_{r}}>0$.
If not, by Roth's Theorem \ref{roth} for any $\epsilon>0$
the right hand side of (\ref{eq:huhn}) is at least 
$M^{-N(1+\epsilon)}=\zeta^{-LN(1+\epsilon)}$ as $N\to\infty$.
In either case taking logarithms according to (\ref{eq:hauser}) and observing
$\frac{n}{N}=L+\frac{r}{N}\leq L+\frac{L-1}{N}$ tends to $L$ as $N\to\infty$ (or equivalently $n\to\infty$)
proves $1$ to be an upper bound for $\overline{\sigma}(\alpha,\zeta)$. 
 The last claim concerning $\widehat{\overline{\sigma}}(\alpha,\zeta)$ admits an analog proof.
\end{proof}

Now we turn to the case of $\zeta$ not of the form $\sqrt[L]{M}$.
We will need a basic property of algebraic field extensions.

\begin{lemma} \label{modernehex}
Let $[K:L]$ be a field extension of algebraic number fields $K,L$.
Then  there exist exactly $[K:L]$ monomorphisms $\tau_{j}:K\mapsto L$ 
which are the identity restricted to $L$.

If $K:L$ and $L:\mathbb{Q}$ are finite field extensions, then each of
the $[L:\mathbb{Q}]$ monomorphisms $L\mapsto \mathbb{C}$ has 
exactly $[K:L]$ extensions to monomorphisms $K\mapsto \mathbb{C}$.  
\end{lemma}

A Corollary to the famous isomorphism extension Theorem 
additionally using the fact that number fields are perfect, see chapter 9 in \cite{30} for example.
We will first restrict to algebraic integers $\alpha,\zeta$.              

\subsubsection{Special case algebraic integers $\alpha,\zeta$}  \label{ueberspitzen}

\begin{theorem} \label{baldi}
 Let $\alpha> 0,\zeta>1$ be real algebraic integers and $\zeta$ not of the form
$\sqrt[L]{M}$ for integers $M,L$. Put $K:=\mathbb{Q}(\zeta,\alpha)$ and
let $[\mathbb{Q}(\zeta):\mathbb{Q}]=k$ and
$[K:\mathbb{Q}]=[K:\mathbb{Q}(\zeta)]\cdot [\mathbb{Q}(\zeta):\mathbb{Q}]=ks$.
Further let $\theta_{1}=\zeta,\theta_{2},\ldots,\theta_{k}$ be the conjugates
of $\zeta$ in $\mathbb{Q}(\zeta)$ labeled such that
$\vert \theta_{2}\vert\geq \vert \theta_{3}\vert\geq \ldots \geq \vert \theta_{k}\vert$.
Put $\eta_{j}:=\frac{\log \vert\theta_{j}\vert}{\log \zeta}$
for $1\leq j\leq k$ and further
let $\theta_{2},\theta_{3},\ldots,\theta_{m}$ be the conjugates
with $\eta_{j}> 1$ if any, else put $m=1$. 
 Then an upper bound for $\overline{\sigma}(\alpha,\zeta)$ is given by
\begin{equation} \label{eq:hammerout}
 \overline{\sigma}(\alpha,\zeta)\leq s\left(\sum_{j=2}^{m} \eta_{j}+ k-m\right).
\end{equation}
In particular $\overline{\sigma}(\alpha,\zeta)<\infty$.
\end{theorem}

\begin{proof}
Put $K:=\mathbb{Q}(\zeta,\alpha)$ and let $\tau_{1},\tau_{2},\ldots,\tau_{ks}$ 
be the monomorphisms $K\mapsto \mathbb{C}$ from Lemma \ref{modernehex}.
Then $\tau_{j}(\alpha\zeta^{n})= \tau_{j}(\alpha)\tau_{j}(\zeta)^{n}$ are the 
conjugates of $\alpha\zeta^{n}$, $1\leq j\leq ks$.																											
Then $\zeta_{1}:=\tau_{1}(\zeta),\zeta_{2}:=\tau_{2}(\zeta),\ldots, \zeta_{ks}:=\tau_{ks}(\zeta)$
are the conjugates of $\zeta$ in $K$ 
and similarly $\alpha_{1}:=\tau_{1}(\alpha),\ldots,\alpha_{k}:=\tau_{ks}(\alpha)$ 
are the conjugates of $\alpha$ in $K$. 

Label such that $\tau_{1}$ is the identity map and
 $\vert \zeta_{2}\vert\geq \vert \zeta_{3}\vert\geq \ldots \geq \vert \zeta_{ks}\vert$.
 Put $\eta_{j}:=\frac{\log \vert\zeta_{j}\vert}{\log \zeta}$
for $1\leq j\leq ks$. Let $\zeta_{2},\zeta_{3},\ldots,\zeta_{r}$ be the conjugates
with $\eta_{j}> 1$ if any, else put $r=1$. We first show

\begin{equation} \label{eq:hammeroutfit}
 \overline{\sigma}(\alpha,\zeta)\leq \sum_{j=2}^{r} \eta_{j}+ ks-r.
\end{equation}

We proceed similar to the proof of the right hand side of (\ref{eq:malina}).
Again put $M_{n}:=<\alpha\zeta^{n}>\in{\mathbb{Z}}$ and consider the polynomials 
\[
 \Phi(M_{n}):= \prod_{j=1}^{ks} \left(M_{n}-\alpha_{j}\zeta_{j}^{n}\right)
\]
Note that again by our assumptions $\Phi(M_{n})\neq 0$ for $n$ sufficiently large.
Indeed, the linear factor $(M_{n}-\alpha_{1}\zeta_{1}^{n})=(M_{n}-\alpha\zeta^{n})\neq 0$
 for any $n$ by our restrictions and Proposition \ref{elendiglich}.
But if another linear factor would be $0$ for arbitrarily large values of $n$,
clearly neither $\vert \zeta_{j}\vert< \vert \zeta\vert$ nor 
 $\vert \zeta_{j}\vert> \vert \zeta\vert$ is possible
due to $\vert M_{n}-\alpha\zeta^{n}\vert \leq 1$ but the difference $\alpha\zeta^{n}-\alpha_{j}\zeta_{j}^{n}$
would clearly tend to infinity by absolute value.
But if $\vert \zeta_{j}\vert =\vert \zeta\vert=\zeta$ and $\alpha_{j}\neq \alpha$, a similar
argument applies. In the remaining case $\vert \zeta_{j}\vert =\vert \zeta\vert=\zeta$ 
and $\alpha_{j}=\alpha$, we have 
$M_{n}\neq \alpha\zeta^{n}=\vert \alpha \zeta^{n}\vert = \vert \alpha_{j}\zeta_{j}^{n}\vert$,
so this case leads to a contradiction too.

Since $\alpha,\zeta$ are in the ring of algebraic integers, similarly to (\ref{eq:franzr}) we obtain
\begin{equation} \label{eq:franzjaeg}
 \left\Vert \alpha\zeta^{n}\right\Vert = \vert M_{n}-\alpha\zeta^{n}\vert
\geq \frac{1}{\prod_{j=2}^{ks} \left\vert M_{n}-\alpha_{j}\zeta_{j}^{n}\right\vert}.
\end{equation}
By $0<M_{n}\leq \alpha\zeta^{n}+1$ the factors in the right hand side denominators
of (\ref{eq:franzjaeg}) are bounded 
above by $\max(\alpha,\vert\alpha_{j}\vert)\left(\vert \zeta\vert^{n}+\vert \zeta_{j}\vert^{n}\right)+1$.
For $n\to\infty$ and with $K:=\max_{2\leq j\leq ks} \vert\alpha_{j}\vert$ this gives
\[
 \left\Vert \alpha\zeta^{n}\right\Vert\geq \frac{1}{2K\cdot \prod_{2\leq j\leq ks}\max(\vert \zeta_{j}\vert,\zeta)^{n}}.
\]
Separating cases $\vert\zeta_{j}\vert>\zeta$ and $\vert \zeta_{j}\vert \leq \zeta$ and
 taking logarithms to base $\zeta$ proves (\ref{eq:hammeroutfit}).

To derive (\ref{eq:hammerout}) from (\ref{eq:hammeroutfit}), it suffices to observe
that by Lemma \ref{modernehex}, each $\theta_{i}$ for $i\in{\{1,2,\ldots,k\}}$
equals $\zeta_{j}$ for exactly $s$ values of $j\in{\{1,2,\ldots,ks\}}$. 
\end{proof}

\begin{corollary} \label{kolmar}
Let $\zeta$ be a Pisot number and $\alpha\neq 0$ be real algebraic of degree \\ $[\mathbb{Q}(\alpha):\mathbb{Q}]=t$.
Then $\overline{\sigma}(\alpha,\zeta)\leq t(k-1)$.
\end{corollary}

\begin{proof}
For Pisot numbers, by their definition $m=1$ in Theorem \ref{baldi} 
and $\sum_{j=2}^{k}\eta_{j}$ is the empty sum.
Moreover, 
\begin{equation} \label{eq:typog}
s=[\mathbb{Q}(\alpha,\zeta):\mathbb{Q}(\zeta)]\leq [\mathbb{Q}(\alpha):\mathbb{Q}]=t
\end{equation}
by an easy vector space argument. 
In view of this the result is immediate due to Theorem \ref{baldi}.
\end{proof}

If $\zeta$ is not a Pisot number, the quantities $\eta_{j}$ in
Theorem \ref{baldi} are somehow annoying.
We want to give a version of Theorem \ref{baldi} where the 
upper bound only depends on the degree of $\alpha$ and
the complexity of $\zeta$, i.e. its degree and the largest
absolute value of the coefficients of its minimal polynomial. 
For this we need Landau's bound for the roots of a polynomial,
which is given in the next Proposition.

\begin{proposition} \label{propL}
Let $z_{1},z_{2},\ldots,z_{k}$ be the $k$ roots of the polynomial 
$P(X)=a_{k}X^{k}+a_{k-1}X^{k-1}+\cdots +a_{0}$ with integer coefficients.
Then
\[
 M(P):=\vert a_{k}\vert \prod_{j=1}^{k} \max(1,\vert \zeta_{j}\vert)\leq \sqrt{\sum_{j=0}^{k}\vert a_{j}\vert^{2}}\leq \sqrt{k+1}H(P).
\]
\end{proposition}

See \cite{13} for a proof. 
Now we can present a Theorem which in fact is a Corollary to Theorem \ref{baldi}.

\begin{theorem} \label{katzchen}
Let $\alpha> 0,\zeta>1$ be real algebraic integers and $\zeta$ not of the form
$\sqrt[L]{M}$ for integers $M,L$. Further let $[\mathbb{Q}(\zeta):\mathbb{Q}]=k$
 be the degree of $\zeta$ and $[\mathbb{Q}(\alpha):\mathbb{Q}]=t$ be 
the degree of $\alpha$.
Further let $N(\zeta)$ be the number of conjugates of $\zeta$ in the
closed unit circle. 
Then we have
\begin{equation} \label{eq:erste}
  \overline{\sigma}(\alpha,\zeta)\leq [\mathbb{Q}(\alpha,\zeta):\mathbb{Q}(\zeta)]\cdot \left(N(\zeta)+
	\frac{\frac{1}{2}\log k+\log H(\zeta)}{\log \zeta}\right).
\end{equation}
In particular
\begin{equation}  \label{eq:zweite}
  \overline{\sigma}(\alpha,\zeta)\leq t\cdot \left(k-1+\frac{\frac{1}{2}\log k+\log H(\zeta)}{\log \zeta}\right).
\end{equation}
\end{theorem}

\begin{proof}
By Theorem \ref{baldi} it suffices to prove the 
 value in the brackets of (\ref{eq:erste}) is not smaller 
than the one in (\ref{eq:hammerout}).

In order to prove this, we look at the roots $\zeta_{j}$ with $\vert \zeta_{j}\vert >1$
and those with $\vert \zeta_{j}\vert\leq 1$ separately. The sum in (\ref{eq:hammerout})
over the latter can be estimated above by their cardinality $N(\zeta)$, as by $\zeta>1$
they in particular have smaller modulus than $\zeta$. For the first type 
we may apply Proposition \ref{propL} to obtain
\[
 \sum_{j=2}^{m} \frac{\log \vert \zeta_{j}\vert}{\log \zeta}
\leq \sum_{j=2}^{k} \max\left(0,\frac{\log \vert\zeta_{j}\vert}{\log \zeta}\right)
\leq \frac{\log(\sqrt{k}H)}{\log \zeta}=\frac{\frac{1}{2}\log k+\log H}{\log \zeta}.
\]
Combining the estimates (\ref{eq:erste}) follows immediately. 
For (\ref{eq:zweite}) further note that again since $\zeta>1$ we have
 $N(\zeta)\leq k-1$ and moreover (\ref{eq:typog}) still holds.
\end{proof}

We finally want to give bounds for arbitrary algebraic integers
depending only on their complexity, i.e. replace the quantity $\log \zeta$
by an expression which solely depends on the degrees of $\alpha,\zeta$ such as $H(\zeta)$.
Under the additional restriction $\zeta>1+\epsilon$ for some $\epsilon>0$ we can just replace 
$\log \zeta$ by $\log (1+\epsilon)$ in the denominator of (\ref{eq:zweite}).

In order to give bounds in the general case we need some more preparation. 

\begin{proposition}  \label{propA}
 Let $m,H$ be positive integers and $x_{1},x_{2},\ldots,x_{m}$ be positive reals
 satisfying the following:
the product of the $x_{j}$ with $x_{j}\geq 1$ is bounded above by $\sqrt{m+1}H$.
Then 

\begin{equation} \label{eq:tsarein}
\prod_{j=1}^{m} (x_{j}+1)\leq 2^{m}\sqrt{m+1}H.
\end{equation}
    
\end{proposition}

\begin{proof}
The claim obviously holds if all $x_{j}\leq 1$ 
so we can assume there exists at least one $x_{j}>1$. 
For any given $x_{1},x_{2},\ldots,x_{m}$ if necessary relabel $x_{j}$ to be 
increasing and define $f$ by
 $x_{1}\leq x_{2}\leq \cdots \leq x_{f}\leq 1< x_{f+1}\leq x_{f+2}\leq \cdots\leq x_{m}$
and $f=0$ in case of all $x_{j}>1$. 
Clearly
\begin{equation} \label{eq:pippn}
\prod_{j=1}^{f} (x_{j+1}+1)\leq 2^{f}.
\end{equation}
To estimate the product over the remaining $x_{j}$, expand $\prod_{j=f+1}^{m} (x_{j}+1)$.
Since $x_{j}\geq 1$ in this case, each of the $2^{m-f}$ products of the arising sum                   
is at most the total product $\prod_{j=f+1}^{m} x_{j}$, which is bounded
by above by $\sqrt{m+1}H$ by our assumption. Thus 
\begin{equation} \label{eq:pipp}
\prod_{j=f+1}^{m} (x_{j+1}+1)\leq 2^{m-f}\sqrt{m+1}H.
\end{equation}
Combining (\ref{eq:pipp}),(\ref{eq:pippn}) gives

\[
 \prod_{j=1}^{m} (x_{j}+1)=\prod_{j=1}^{f} (x_{j}+1)\prod_{j=f+1}^{m}(x_{j}+1)
\leq 2^{f}\cdot 2^{m-f}\sqrt{m+1}H= 2^{m}\sqrt{m+1}H.
\] 
\end{proof}

\begin{corollary} \label{quark}
For $\zeta$ as in Theorem \ref{baldi} or \ref{katzchen} 
we have 

\[
0<\log \zeta\leq \log\left(1+\frac{1}{2^{k-1}\sqrt{k}H(\zeta)}\right).
\] 

\end{corollary}

\begin{proof}
Consider $P(1)=\prod_{j=1}^{k}(1-\zeta_{j})=(1-\zeta)\prod_{j=2}^{k}(1-\zeta_{j})$. 
Due to the irreducibility of $P$ this is a nonzero integer, so we have
\[
\zeta-1\geq \frac{1}{\prod_{j=2}^{k}\vert 1-\zeta_{j}\vert}\geq \frac{1}{\prod_{j=2}^{k}(1+\vert \zeta_{j}\vert)}.
\]
We can apply Proposition \ref{propA} to the right hand side
with $m:=k-1$ and $x_{j}:=\vert \zeta_{j+1}\vert$
for $1\leq j\leq m$ and $H:=H(\zeta)$, because its assumption 
is satisfied by Proposition \ref{propL}, which leads to 
\begin{equation*}                                      
\zeta-1\geq 2^{-(k-1)}\frac{1}{\sqrt{k}H(\zeta)}.
\end{equation*}
Adding one and taking logarithms finishes the proof.  
\end{proof}

As an immediate Corollary to Theorem \ref{katzchen} and Corollary \ref{quark} we get

\begin{theorem}  \label{quargel}
Let $k,t,H$ be positive integers. 
For all real algebraic integers $\zeta>1$ of degree $[\mathbb{Q}(\zeta):\mathbb{Q}]\leq k$ and $H(\zeta)\leq H$
and all non-zero real algebraic integers $\alpha$ of degree $[\mathbb{Q}(\alpha):\mathbb{Q}]\leq t$ 

\begin{equation}  \label{eq:vierte}
\overline{\sigma}(\alpha,\zeta)\leq t\cdot 
\left(k-1+\frac{\frac{1}{2}\log k+\log H}{\log\left(1+ \frac{1}{2^{k-1}\sqrt{k}H}\right)}\right)  
\end{equation}

holds.                                                
\end{theorem}

\begin{corollary} \label{najut}
 With the notation and assumptions of Theorem \ref{quargel} 

\[
\overline{\sigma}(\alpha,\zeta)\leq
 t\left(k-1+2^{k}\sqrt{k}H\left(\frac{1}{2}\log k+\log H\right) \right)
\] 

holds.
\end{corollary}

\begin{proof}
Clearly $2^{k-1}\sqrt{k}H\geq 1$, so its reciprocal is in $(0,1)$.
So looking at the denominator in (\ref{eq:vierte}), it suffices to prove
$\log (1+x)-\frac{x}{2}>0$ for $0<x<1$. However, this is easily checked. 
\end{proof}

\subsubsection{General case} \label{subsek}

To obtain bounds for $\overline{\sigma}(\alpha,\zeta)$ 
we put the general case down to the special case of algebraic integers. 
It is well known that the algebraic integers are the fraction field
of the ring of algebraic integers. The classical way to prove this
is for any algebraic number $z$ of degree $[\mathbb{Q}(z):\mathbb{Q}]=k$
with minimal polynomial $P(X)=a_{k}X^{k}+a_{k-1}X^{k-1}+\cdots +a_{0}$,  
i.e.
\[
a_{k}z^{k}+a_{k-1}z^{k-1}+\cdots +a_{0}=0 \quad \Leftrightarrow 
\quad a_{k}^{k-1}\left(a_{k}X^{k}+a_{k-1}X^{k-1}+\cdots +a_{0}\right)=0
\]
to observe that $z^{\prime}:=a_{k}z$ is an algebraic integer
of the same degree because it is a root of the monic polynomial
\[
 Q_{P}(X)= X^{k}+a_{k-1}X^{k-1}+a_{k-2}a_{k}X^{k-2}+a_{k-3}a_{k}^{2}X^{k-3}+\cdots +a_{0}a_{k}^{k-1}=0.
\]
(Note: The degree of $z^{\prime}$ cannot be strictly less 
than $k$ since any polynomial equation $Q(z^{\prime})=0$ with $\rm{deg}(Q)<k$
 in $z^{\prime}$ can be written as a polynomial equation $R(z)=0$ with $\rm{deg}(Q)=\rm{deg}(R)<k$,
 contradicting the minimality of $k$).
This construction additionally gives the bounds 
\begin{equation}  \label{eq:hoehe}
\vert z^{\prime}\vert \leq H(z)\vert z\vert, \qquad H(z^{\prime})\leq H(z)^{k}.
\end{equation}
Now in view of Proposition \ref{interessant} we can generalize our results
from the last subsection \ref{ueberspitzen}.

\begin{theorem} \label{thero}
Let $k,t,H$ be positive integers. 
For all real algebraic numbers $\zeta>1$ not of the form 
$\frac{\sqrt[L]{M_{0}}}{N_{0}}$ for $L,M_{0},N_{0}\in{\mathbb{Z}}$
of degree $[\mathbb{Q}(\zeta):\mathbb{Q}]\leq k$ and $H(\zeta)\leq H$
and all non-zero real algebraic numbers $\alpha$ of degree $[\mathbb{Q}(\alpha):\mathbb{Q}]\leq t$ 

\begin{equation} \label{eq:terrorist}
 \overline{\sigma}(\alpha,\zeta)\leq 2^{k}\sqrt{k}t\cdot H\log H \cdot 
\left(2k-1+\frac{\frac{1}{2}\log k+ \log H}{\log 2}\right)+2^{k}\sqrt{k}H\log H-1
\end{equation}

holds. In particular $\overline{\sigma}(\alpha,\zeta)<\infty$.
\end{theorem}

\begin{proof}
 Let $\alpha,\zeta$ be arbitrary real algebraic numbers satisfying the assumption.
Observe that by $\zeta>1$ and distinguishing cases $\overline{\sigma}(\alpha,\zeta)\log \zeta-\log N>0$
and $\overline{\sigma}(\alpha,\zeta)\log \zeta-\log N\leq 0$, 
Proposition \ref{interessant} can be written as
\begin{equation} \label{eq:feirefitz}
 \overline{\sigma}(\alpha,\zeta)\leq 
\max\left\{\frac{\overline{\sigma}(M\alpha,N\zeta)(\log \zeta+\log N)+\log N}{\log \zeta},\frac{\log N}{\log\zeta}\right\}.
\end{equation}
By (\ref{eq:joehklar}) applied to $M\alpha,N\zeta>N\geq 1$,
the left expression in the maximum is the larger 
one and it can be written as
\begin{equation} \label{eq:by}
\frac{\overline{\sigma}(M\alpha,N\zeta)(\log \zeta+\log N)+\log N}{\log \zeta}
= \overline{\sigma}(M\alpha,N\zeta)\left(1+\frac{\log N}{\log \zeta}\right)+\frac{\log N}{\log \zeta}.
\end{equation}
Denote $\zeta^{\prime},\alpha^{\prime}$ the algebraic integers
arising from $\zeta,\alpha$ by the above construction, i.e.
$M:=b_{t}$ resp. $N:=a_{k}$ are the leading coefficients
of the minimal polynomials of $\alpha$ resp. $\zeta$ and 
$\zeta^{\prime}=a_{k}\zeta, \alpha^{\prime}=b_{t}\alpha$. 
We may assume $N\geq 2$. Indeed, if else $N=1$ then $\zeta$ is an algebraic integer
and since furthermore $M\alpha$ has the same degree as $\alpha$,
the upper bounds from Corollary \ref{najut} of subsection \ref{ueberspitzen} 
are valid for the right hand side of (\ref{eq:by}),
which are better then the one in (\ref{eq:terrorist}). (Anyway, if $\zeta$
is an algebraic integer nevertheless we can apply the results
of subsection \ref{ueberspitzen} to $N\zeta$ for any $N$, so really $N\geq 2$ is no restriction.)

We may apply (\ref{eq:zweite}) from Theorem \ref{katzchen} to $\alpha^{\prime},\zeta^{\prime}$.
In view of (\ref{eq:hoehe}), i.e $N\leq H(\zeta)\leq H$ and $H(N\zeta)\leq N^{k}H(\zeta)\leq N^{k}H$,
and $N\zeta>N\geq 2$ this gives
\begin{eqnarray}
 \overline{\sigma}(M\alpha,N\zeta)&\leq& t\cdot \left(k-1+\frac{\frac{1}{2}\log k+\log H(N\zeta)}{\log \zeta}\right) \nonumber \\
 &\leq& t\cdot \left(k-1+\frac{\frac{1}{2}\log k+\log H(N\zeta)}{\log N}\right)         \nonumber                \\
 &\leq& t\cdot \left(2k-1+\frac{\frac{1}{2}\log k+\log H}{\log N}\right)                \nonumber                      \\
 &\leq& t\cdot \left(2k-1+\frac{\frac{1}{2}\log k+\log H}{\log 2}\right).        
\label{eq:cleopatra}                                                
\end{eqnarray}
It remains to find lower bounds for $\log \zeta$.
However, Corollary \ref{quark} holds in case of arbitrary algebraic numbers 
$\zeta$ too with an almost analogue proof. Indeed, similarly
 \[
\zeta-1\geq \frac{1}{\vert a_{k}\vert}\cdot \frac{1}{\prod_{j=2}^{k} (\vert a_{j}\vert+1)},
\]
but the bound from Proposition \ref{propL} is better by 
the same factor $\frac{1}{\vert a_{k}\vert}$ making Proposition \ref{propA}
applicable with $H$ replaced by $\frac{H}{\vert a_{k}\vert}$,
thus giving the same result. Combining the mentioned bound from Corollary \ref{quark} 
with the easy estimate (similar to the proof of Corollary \ref{najut})
\[
1+\frac{\log N}{\log(1+x)}\leq \log N\left(\frac{1}{\log(1+x)}+1\right)+1-\log N<\log N\frac{2}{x}+1-\log N, \qquad x\in{\left(0,\frac{1}{2}\right]},
\]
for $x:=\zeta-1\leq \frac{1}{2}$ (otherwise the bounds are much better anyway!), we obtain 
\begin{equation} \label{eq:kleopatra}
 1+\frac{\log N}{\log \zeta}\leq 2^{k}\sqrt{k}H\log N+1-\log N\leq 2^{k}\sqrt{k}H\log H.
\end{equation}
Inserting (\ref{eq:kleopatra}) and (\ref{eq:cleopatra}) in (\ref{eq:by}) 
 leads to the desired upper bound in (\ref{eq:terrorist}) for (\ref{eq:by}), which
is the left (and larger) expression in the maximum of (\ref{eq:feirefitz}). 
\end{proof}

\begin{remark} 
Note that again the additional assumption $\zeta>1+\epsilon$ for some $\epsilon>0$ improves 
the bounds concerning $H$ from order $H\log H^{2}$ to order $\log H$.
\end{remark}

\begin{remark}
Fix $k,t$ in Theorem \ref{thero} or Corollary \ref{najut}
and let $H$ be large. The bounds become
\[
 \overline{\sigma}(\alpha,\zeta)\leq C(k,t)\cdot H(\log H)^{2}
\]
for a constant $C(k,t)$, where the square can be dropped 
in case of algebraic integers. Note that noticeable improvements of this
bound imply improvements of Theorem \ref{baldi}.
Indeed, the monic irreducible quadratic polynomials $P_{2,H}(X)=X^{2}-HX+H$ have $H(P)=H$
and the roots satisfy $\zeta_{1}\thicksim 1+\frac{1}{2H}, \zeta_{2}\thicksim H$
as $H\to\infty$. Thus the bracket expression in the right hand side of
(\ref{eq:hammerout}) is of order $C(2,s)H\log H$ with $C(2,s)=2$. 

Conversely, for fixed $t,H$ and $k\to\infty$ the bounds seem not to be optimal
due to possible improvements of Corollary \ref{quark}.
According to chapter 3 in \cite{15} the closest root $z(k)$ to $1$ among all polynomials
$P_{k}$ of degree $k$ with coefficients among $\{-1,0,1\}$  
 and $P_{k}(1)\neq 0$ is of order $\vert z(k)-1\vert\asymp \frac{1}{k^{2}}$.
In chapter 4 of \cite{15} the polynomials $R_{k}$ with the respective roots $z(k)$ closest to $1$ 
are explicitly given as 
\begin{eqnarray*}
R_{k}&=& \pm  \frac{x^{2m+1}-x^{2m}+1}{x-1}, \qquad k=2m         \\
R_{k}&=& \pm  \frac{x^{2m+2}-x^{m+1}-x^{m}+1}{x-1}, \qquad k=2m+1.
\end{eqnarray*}
So for $H=1$ and all $\zeta$ with $[\mathbb{Q}(\zeta):\mathbb{Q}]\leq k$
these estimates indeed improve the bound in Proposition \ref{quark} and consequently in
Corollary \ref{najut} as for $k\to\infty$ we have
$\frac{1}{\log \zeta}\asymp k^{2}<2^{k}\sqrt{k}H=2^{k}\sqrt{k}$. 

For general $H$, the minimal value of $\zeta-1$ for roots of polynomials $P$ 
of arbitrary degree with $H(P)\leq H$ was investigated
in \cite{17}, \cite{18}, \cite{19}, \cite{20}, \cite{21} often in terms of the Weil height
of a polynomial.
 However no result seems to be directly applicable to derive a much better bound.

We want to quote a last result where the distance of algebraic numbers to an integer
converges indeed exponentially to $0$ with increasing degree.
In \cite{16} it is shown that the closest distance $\tau\neq 0$ of the root of any polynomial
$P$ of degree $\leq k$ and $H(P)\leq H$ from any integer is either $0$ or
bounded by $H(H+1)^{-k}<\tau<2H(H+1)^{-k}$ for $k$ sufficiently large
and the minimum is taken for a root of the polynomial $S_{k}:=X^{k}-H(X^{k-1}+X^{k-2}+\ldots X+1)$.
 In particular this result applies to our concern: For any algebraic number $\beta$ 
of degree $[\mathbb{Q}(\beta):\mathbb{Q}]\leq k$ and $H(\beta)\leq H$ the distance $\tau$ 
to {\em any} integer satisfies $H(H+1)^{-k}<\tau$ for $k$ sufficiently large.
Moreover, provided that $S_{k}$ is irreducible, which is true by Eisenstein criterion if 
$H$ has a prime divisor $p$ with $p^{2}\nmid H$,
an upper bound for $\tau$ is given too by $H(H+1)^{-k}<\tau<2H(H+1)^{-k}$.
 Clearly, these bounds in \cite{16} are no improvement of our results in any way,
 however replacing $1$ by any integer weakens the statement a lot. 
 We will not go into deeper detail here since for our concern the finiteness
 of $\overline{\sigma}(\alpha,\zeta)$ in the algebraic case is most important.                         
\end{remark}

The case of $\zeta=\sqrt[L]{\frac{p}{q}}$ with $q>1$ is still missing.
Similar to Proposition \ref{elendiglich} Roth's Theorem \ref{roth} 
applies to this case.

\begin{proposition} \label{netzhex}
Let $\zeta=\sqrt[L]{\frac{p}{q}}>1$ with relatively prime positive integers $p>q>1$ 
and $L$ minimal with this property. 
For any real algebraic $\alpha\neq 0$ we have
$\overline{\sigma}(\alpha,\zeta)\leq \frac{\log p + \log q}{\log p - \log q}$.
In particular $\overline{\sigma}(\alpha,\zeta)<\infty$.
\end{proposition}

\begin{proof}
Write $n=LN+r$ with integers $N,0\leq r\leq L-1$ such that 
\[
\Vert \alpha\zeta^{n}\Vert= \left\Vert \left(\frac{p}{q}\right)^{N}(\alpha\zeta^{r})\right\Vert.
\]
If $\alpha \zeta^{r}$ is rational, so is the expression above.
However, it is not $0$ for all sufficiently large $N$ by our assumption
$q>1$ and Proposition \ref{tuertor}. Hence in this case
\begin{equation} \label{eq:melken}
\Vert \alpha\zeta^{n}\Vert \geq \frac{1}{Kq^{N}}
\end{equation}
for some constant $K$, which can be chosen the maximum of the denominators of the rational
$\alpha\zeta^{r}$ if there exist any.
For those $r$ such that $\alpha\zeta^{r}\notin{\mathbb{Q}}$, by Roth's Theorem \ref{roth}
for any $\epsilon>0,T\in{\mathbb{Z}}$ and $N=N(\epsilon)$ sufficiently large
\[
\left\vert \alpha\zeta^{r}p^{N}-q^{N}T\right\vert \geq p^{-N(1+\epsilon)},
\]
so dividing by $q^{N}$
\begin{equation} \label{eq:melenk}
\left\Vert \left(\frac{p}{q}\right)^{N}(\alpha\zeta^{r}) \right\Vert \geq p^{-N(1+\epsilon)}q^{-N}.
\end{equation}
The bound in (\ref{eq:melenk}) is clearly worse than
the one in (\ref{eq:melken}) for large $N$.
 
Letting $\epsilon\to 0$, taking logarithms in (\ref{eq:melenk}) 
according to (\ref{eq:hauser}) and observing
 $\frac{n}{N}\thicksim L$ for $n\to\infty$ yields 
the assertion.  
\end{proof}

\begin{remark}
Note that the minimal polynomial
$P_{L}(X):=qX^{L}-p$ of $\zeta$ is of height $H(P_{L})=H(\zeta)=p$. Thus
$\overline{\sigma}(\alpha,\zeta) \leq \frac{2\log H(\zeta)}{L\log \zeta}$
and that $L$ is just the degree of $\zeta$. For fixed, $L,H$ and $H(\zeta)=p\leq H$
we have $\zeta^{L}\geq \frac{p}{p-1}$, so this becomes
\[
\overline{\sigma}(\alpha,\zeta) \leq \frac{2\log H}{\log H-\log (H-1)}\thicksim 2H\log H,
\]
similar to the results in the case $\zeta\neq \sqrt[L]{\frac{p}{q}}$ but
without appearance of the degrees $k=L,t$ of $\zeta,\alpha$.

Again restricting to $\zeta>1+\epsilon$ for some $\epsilon>0$, due
to $\log p-\log q= L\cdot \log \zeta$ improves the bounds
to order $C\cdot \frac{\log H}{L}$ for a constant $C=C(\epsilon)=\frac{2}{\log(1+\epsilon)}$
depending on $\epsilon$ only, which even seems to improve as the degree $L$ of $\zeta$ increases.
Note however, that $\zeta_{L}=\sqrt[L]{\frac{p_{L}}{q_{L}}}>1+\epsilon$
 implies $H(\zeta_{L})=p_{L}\geq \zeta_{L}^{L}$ tends to infinity
exponentially as $L\to\infty$, so in fact it does not. 
\end{remark}

We want to point out a particular consequence of the combination of
Theorem \ref{thero} and Proposition \ref{netzhex} that can be interpreted as a transcendence criterion.

\begin{theorem} \label{rantfrit}
For any real algebraic $\alpha\neq 0$ and $\zeta>1$ not of the form
 $(\alpha,\zeta)=\left(\frac{A}{B}M^{-\frac{g}{L}},M^{\frac{1}{L}}\right)$ 
for integers $L,M,A,B,g>0$ we have $\overline{\sigma}(\alpha,\zeta)<\infty$. In particular,
for any algebraic $\zeta>1$ not of the form $\sqrt[L]{M}$ for integers $L,M$ and any real algebraic 
$\alpha\neq 0$ we have $\overline{\sigma}(\alpha,\zeta)<\infty$.
\end{theorem}

Note that by (\ref{eq:portugues}) we can drop the assumption $\zeta>1$ in Theorem \ref{rantfrit}. 
It is likely that in fact $1$ is the correct uniform upper bound, which we want to state as a Conjecture. 

\begin{conjecture}  \label{confituer}
For any real algebraic $\alpha\neq 0$ and real $\zeta$ not of the form
 $(\alpha,\zeta)=\left(\frac{A}{B}M^{-\frac{g}{L}},M^{\frac{1}{L}}\right)$ 
for integers $L,M,A,B,g>0$ we have $\overline{\sigma}(\alpha,\zeta)\leq 1$. In particular,
For any real algebraic $\zeta$ not of the form $\sqrt[L]{M}$ for integers $L,M$ and any real algebraic 
$\alpha\neq 0$ we have $\overline{\sigma}(\alpha,\zeta)\leq 1$.
\end{conjecture}

Note that $1$ is the best bound possible because algebraic Pisot units of degree $2$ have
$\sigma_{n}(1,\zeta)=1$ for all $n\geq 1$ by Theorem \ref{kaffee}. 

As a final Remark to the case of algebraic numbers $\alpha,\zeta>1$
we want to point out that Roth's Theorem suggests that the set $\overline{\sigma}(\alpha,\zeta)>0$
should be somehow small. In this case of for $\epsilon>0$ certain arbitrarily large values of $n$
\[
 \left\Vert\zeta^{-n}\scp{\zeta^{n}}\right\Vert = \left\Vert\zeta^{-n}(\zeta^{n}+\Vert \zeta^{n}\Vert)\right\Vert
=\left\Vert\zeta^{-n}\Vert \zeta^{n}\Vert\right\Vert \leq \zeta^{-n(1+\epsilon)}.
\]
So by $\scp{\zeta^{n}}\leq \zeta^{n}+1$ for large $n$
\[
\left\Vert\zeta^{-n}K\right\Vert \leq K^{-1-\frac{\epsilon}{2}}
\]
has the solution $K:= \scp{\zeta^{n}}$.
By Roth's Theorem, for fixed $n$ there are only finitely many $K$ satisfying this
equation. However, as there is only one solution this is just a heuristic 
argument and indeed Pisot numbers provide counterexamples where $\overline{\sigma}(\alpha,\zeta)>0$.

\section{Results for $\underline{\sigma}(\alpha,\zeta),\overline{\sigma}(\alpha,\zeta)$ in the general case}    \label{jjj}

If $\zeta$ is transcendental, then as in section \ref{iii} 
Theorem \ref{gans} applies, so $\underline{\sigma}(\alpha,\zeta)=0$.
We nevertheless want to give a different proof for this fact, relying only
on Diophantine approximation properties and avoiding Fourier analysis.

\begin{theorem}   \label{schmidtbenutzt}
 Let $\zeta>1$ be a transcendental real number. Then $\underline{\sigma}(\alpha,\zeta)=0$ for all $\alpha\neq 0$.
\end{theorem}

\begin{proof}
 If else $\underline{\sigma}(\alpha,\zeta)=\epsilon>0$ for $\zeta$ transcendental real and
$\alpha\neq 0$ an arbitrary real number, for arbitrary but fixed $k$ and any 
$n\geq n_{0}(k)$ sufficiently large we would have
\begin{eqnarray*}
 \Vert \alpha\zeta^{n}\Vert&\leq&  \zeta^{-n\frac{\epsilon}{2}}  \\
 \Vert \alpha\zeta^{n+1}\Vert=\Vert \zeta \cdot \alpha\zeta^{n}\Vert&\leq&  \zeta^{-n\frac{\epsilon}{2}} \\
 \vdots \qquad \vdots \qquad \vdots  \\
 \Vert \alpha\zeta^{n+k}\Vert=\Vert \zeta^{k} \cdot \alpha\zeta^{n}\Vert&\leq& \zeta^{-n\frac{\epsilon}{2}}.
\end{eqnarray*}
Thus putting $M_{n}:=<\alpha\zeta^{n}>$ and using the first equation in the others gives by triangular inequality
\begin{equation} \label{eq:phantom}
 \left\Vert M_{n}\zeta^{j}\right\Vert \leq (\zeta^{j}+1)\zeta^{-n\frac{\epsilon}{2}}
\leq 2\zeta^{k}\cdot \zeta^{-n\frac{\epsilon}{2}}, \qquad 2\leq j\leq k.
\end{equation}
Now look at the quantity $\widehat{\lambda}_{k}(\boldsymbol{\zeta})$ 
with $\boldsymbol{\zeta}:=(\zeta,\zeta^{2},\ldots,\zeta^{k})$ for the present $\zeta$.
In Definition \ref{defibrillator} of the quantities $\widehat{\lambda}_{d}$  
for any $X$ taking $x=M_{n(X)}=\Vert\alpha\zeta^{n(X)}\Vert$ with $n(X)$ the largest integer $n$  
such that $\alpha\zeta^{n}\leq X$, and noting that $1\leq \frac{X}{M_{n(X)}}\leq \zeta$
by (\ref{eq:phantom})
we would have $\widehat{\lambda}_{k}(\boldsymbol{\zeta})\geq \frac{\epsilon}{2}>0$.
However, as this is valid for any positive integer $k$ and $\zeta$ is transcendental, 
this contradicts Theorem \ref{joxe}.
So the assumption $\underline{\sigma}(\alpha,\zeta)>0$ cannot hold.  
\end{proof}

For sake of completeness we give an immediate Corollary.

\begin{corollary}
The set of $\zeta$ such that $\underline{\sigma}(\alpha,\zeta)>0$ for some real $\alpha\neq 0$
{\upshape(}that may depend on $\zeta${\upshape)} equals the set of Pisot numbers.
The set of real $\alpha\neq 0$ such that $\underline{\sigma}(\alpha,\zeta)>0$ for some $\zeta$
{\upshape(}that may depend on $\alpha${\upshape)} consists of
algebraic numbers only. In particular both sets are countable infinite sets. 
\end{corollary}

\begin{proof}
The first assertion follows from the combination of Theorems \ref{pisot}, \ref{schmidtbenutzt}.
The second assertion is immediate due to Theorem \ref{pisot}. Finally it is well-known
that there exist infinitely many Pisot numbers, which can be deduced from Proposition \ref{pisotpol},
and since the second set is obviously closed under $\alpha\mapsto M\alpha$ for integers $M$ by
Proposition \ref{interessant} as well as $\alpha\mapsto \zeta^{k}\alpha$ for an integer $k$ and
a $\zeta$ associated to $\alpha$ it cannot be finite either. 
\end{proof}

Now we turn to the constant $\overline{\sigma}(\alpha,\zeta)$.
First we prepare some results, showing that conversely to $\underline{\sigma}(\alpha,\zeta)$, 
the set of pairs $(\alpha,\zeta)\in{\mathbb{R}^{2}}$ with large values of
$\overline{\sigma}(\alpha,\zeta)$ is not small in sense of cardinality.

\begin{definition}
 For any real $\zeta>1$ and $\delta\geq 0$, denote with $\Delta_{\delta}(\zeta)$ the set of 
$\alpha$ such that $\overline{\sigma}(\alpha,\zeta)> \delta$. In particular let
 $\Delta_{\infty}(\zeta)$ be the set of real numbers $\alpha$ such that 
$\overline{\sigma}(\alpha,\zeta)=\infty$. 
\end{definition}

We start with an easy Proposition discussing the algebraic
structure of the sets $\Delta_{\delta}(\zeta)$.

\begin{proposition}  \label{ganstiehl}
For $\delta_{1}<\delta_{2}$ we have $\Delta_{\delta_{1}}(\zeta)\supset \Delta_{\delta_{2}}(\zeta)$.
For any $\delta\in{[0,\infty]}$ the set $\Delta_{\delta}(\zeta)$ is closed under any map
 $\alpha\mapsto M\zeta^{k}\cdot \alpha$ for any integers $M,k$. 
Furthermore $\Delta_{\delta}(\sqrt[k]{\zeta})\supset \Delta_{\delta}(\zeta)$
for all $\delta\in{[0,\infty]}$ and all integers $k\geq 0$.
\end{proposition}

\begin{proof}
The first point is obvious by the definition of the quantities $\Delta_{\delta}(\zeta)$.
It is also immediate due to its definition that $\Delta_{\delta}$ is closed under $\alpha\mapsto \zeta^{k}\alpha$.
So the other assertions follow easily from Proposition \ref{interessant} with $N=1$
respectively (\ref{eq:nanoleuchte}).  
\end{proof}

Now we aim to give some metrical results.

\begin{theorem}  \label{kommtno}
For any real $\zeta>1$ and any real $b>a$, the set $\Delta_{\infty}(\zeta)\cap (a,b)$ has the same
cardinality as $\mathbb{R}$. In particular $\Delta_{\infty}(\zeta)$ is dense in $\mathbb{R}$. 
\end{theorem}

\begin{proof} 
We explicitly construct $\alpha$ in dependence of $\zeta$. 
Let $n_{1},n_{2},\ldots$ be a strictly monotonically growing sequence of positive integers
to be specified later. Define
\begin{equation} \label{eq:turnher}
 \alpha:= \sum_{j=1}^{\infty} \frac{c_{j}}{\zeta^{n_{j}}}
\end{equation}
where $0\leq c_{j}<1$ are defined recursively in the following way.
Put $c_{1}:= \frac{1}{\zeta^{n_{1}}}$. Note $c_{1}\zeta^{n_{1}}=1\in{\mathbb{Z}}$.
Similarly define $c_{2}$ as the smallest positive real number such that
 $\zeta^{n_{2}}(\frac{c_{1}}{\zeta^{n_{1}}}+ \frac{c_{2}}{\zeta^{n_{2}}})\in{\mathbb{Z}}$,
equivalently
$\zeta^{n_{2}}(\frac{c_{1}}{\zeta^{n_{1}}}+ \frac{c_{2}}{\zeta^{n_{2}}})=
\left\lceil \zeta^{n_{2}}\cdot \frac{c_{1}}{\zeta^{n_{1}}}\right\rceil$,
which gives
\[
 c_{2}:= \left\lceil \zeta^{n_{2}-n_{1}}c_{1}\right\rceil-\zeta^{n_{2}-n_{1}}c_{1}.
\]
In general, having defined $c_{1},c_{2},\ldots,c_{g-1}$ putting 
\[
c_{g}:=  \left\lceil \sum_{j=1}^{g-1}\zeta^{n_{g}-n_{j}}c_{j}\right\rceil-\sum_{j=1}^{g-1}\zeta^{n_{g}-n_{j}}c_{j}
\] 	
we have 
\begin{equation} \label{eq:schuetze}
\zeta^{n_{g}}\sum_{j=1}^{g} \frac{c_{j}}{\zeta^{n_{j}}}\in{\mathbb{Z}}.
\end{equation}
By $0\leq c_{j}<1$,
clearly the sum (\ref{eq:turnher}) converges and $\alpha$ is well defined. Moreover
by (\ref{eq:schuetze}) for any sufficiently large $g$ we have
\[
 \left\Vert \alpha \zeta^{n_{g}}\right\Vert = \left\Vert \zeta^{n_{g}}\sum_{j=1}^{\infty} \frac{c_{j}}{\zeta^{n_{j}}}\right\Vert
 = \left\Vert \zeta^{n_{g}}\sum_{j=g+1}^{\infty} \frac{c_{j}}{\zeta^{n_{j}}}\right\Vert
\leq \sum_{j=g+1}^{\infty} \frac{1}{\zeta^{n_{j}-n_{g}}}.
\]
The assertion $\overline{\sigma}(\alpha,\zeta)=\infty$ follows easily for 
any $\alpha$ arising from a sequence with \\ $\limsup_{j\to\infty} \frac{n_{j+1}}{n_{j}}=\infty$,
for instance $n_{j}=j!$. Clearly, this method is flexible enough to provide card($\mathbb{R}$) many $\alpha$.

To see it is dense we may generalize the construction by putting $c_{1}=\frac{L}{\lambda^{n_{1}}}$
for an arbitrary integer $L$. If we choose $n_{2}-n_{1}$ sufficiently large we can
guarantee $\alpha\in{\left(\frac{L}{\lambda^{n_{1}}},\frac{L+1}{\lambda^{n_{1}}}\right)}$.
 Finally since we can take $n_{1}$ sufficiently large we see the set is dense. Combining
the arguments for density and cardinality, we infer the assertion of the Theorem. 
\end{proof}

The following Theorem \ref{ministerium} shows that the set of $\alpha$
with good approximation properties concerning $\overline{\sigma}(\alpha,\zeta)$ 
for some $\zeta>1$ that may depend on $\alpha$ is small
in measure theoretic sense.  To simplify its proof and some following proofs
at some points we prepend a well-known fact from measure theory first.

\begin{lemma}[Basic measure theory] \label{einzigeslemma}
Let $(\Omega,\mathscr{A},\mu)$ be a measure space and for any $\epsilon>0$ let $M_{\epsilon}\in{\mathscr{A}}$
 be arbitrary $\mu$-measurable subsets of $\Omega$ with the property
 $\epsilon_{1}<\epsilon_{2}\Rightarrow M_{\epsilon_{1}}\supset M_{\epsilon_{2}}$. 
If for any $\epsilon>0$ the set $M_{\epsilon}$ has measure $\mu(M_{\epsilon})=0$,
then $\mu(\cup_{\epsilon>0} M_{\epsilon})=0$ too.
\end{lemma} 

\begin{proof}
By the inclusion property and the sigma subadditivity of measures we have

\[
\mu(\cup_{\epsilon>0} M_{\epsilon})=\mu\left(\cup_{n\geq 1} M_{\frac{1}{n}}\right)
\leq \sum_{n=1}^{\infty} \mu\left(M_{\frac{1}{n}}\right)=\sum_{n=1}^{\infty} 0= 0. 
\] 
\end{proof}

\begin{definition}
 For any $\delta\geq 0$, let $\Delta_{\delta}:=\cup_{\zeta>1} \Delta_{\delta}(\zeta)$, i.e. the set of 
$\alpha$ such that $\overline{\sigma}(\alpha,\zeta)> \delta$ for some $\zeta>1$
that may depend on $\alpha$.
\end{definition}

\begin{theorem}  \label{ministerium}
For any $\delta\in{[0,\infty]}$ the $s$-dimensional Hausdorff measure of $\Delta_{\delta}$ is $0$
for all $s\geq \frac{1}{1+\delta}$. In particular: 
\begin{itemize}
\item the Hausdorff dimension of $\Delta_{\delta}$ is at most $\frac{1}{1+\delta}$
\item the set of $\alpha$ with 
$\overline{\sigma}(\alpha,\zeta)>0$ for any $\zeta>1$ {\upshape(}that my depend on $\alpha${\upshape)} has Lebesgue measure $0$.
\end{itemize}
\end{theorem}

\begin{proof}
We may restrict to $\alpha>0$, since clearly 
$\alpha\in{\Delta_{\delta}}\Leftrightarrow -\alpha\in{\Delta_{\delta}}$.
By Lemma \ref{einzigeslemma} with 
$\Omega=\mathbb{R},\mu=\lambda_{s}$ and $M_{\epsilon}:=\Delta_{\delta}\cap [\epsilon,\frac{1}{\epsilon}]$
it suffices to prove the assertion with the additional restriction $\alpha\in{(A,B)}$ for any given $B>A>0$.
Similarly Lemma \ref{einzigeslemma} implies it is sufficient to restrict to
$\zeta\in{(1+C,\infty)}$ for arbitrary but fixed $C>0$ in the definition of $\Delta_{\delta}$.

So we restrict to this case for arbitrary but fixed $B>A>0,C>0$ in the sequel and denote the resulting set 
with $\Delta_{\delta,A,B,C}$. By construction
$\Delta_{\delta,A,B,C}\subset \Delta_{\delta}\cap (A,B)\subset {\Delta_{\delta}}$. We are left to prove 
\begin{equation} \label{eq:sixfootfour}
\lambda_{s}(\Delta_{\delta,A,B,C})=0, \qquad \forall s\geq \frac{1}{1+\delta}, \quad \forall B>A>0, \quad \forall C>0.
\end{equation}
Fixing $\zeta>1$ for $\alpha\neq 0$ to lie in $\Delta_{\delta+\epsilon}(\zeta)$
for some $\epsilon>0$ there are infinitely many positive integer pairs $n,M_{n}$ with 
$\vert M_{n}-\alpha\zeta^{n}\vert \leq \zeta^{-n(\delta+\epsilon-\frac{\epsilon}{2})}=\zeta^{-n(\delta+\frac{\epsilon}{2})}$,
hence 
\[
\alpha\in{I_{n,M_{n}}:=\left(M_{n}\zeta^{-n}-\zeta^{-n(1+\delta+\frac{\epsilon}{2})},
M_{n}\zeta^{-n}+\zeta^{-n(1+\delta+\frac{\epsilon}{2})}\right)}.
\]
These intervals have length $\lambda_{1}(I_{n,M_{n}})=2\zeta^{-n(1+\delta+\frac{\epsilon}{2})}$, 
but by our restriction $\alpha\in{(A,B)}$ we have $M_{n}\asymp \zeta^{n}$ as $n\to\infty$, 
so the interval lengths are of order $\lambda_{1}(I_{n,M_{n}})\asymp M_{n}^{-(1+\delta+\frac{\epsilon}{2})}$ as $n\to\infty$.

On the other hand, note $\zeta^{-n(1+\delta+\frac{\epsilon}{2})}<1$. Consequently for any fixed integer $M> \max\{A,2B\}$,
such that $2\frac{M}{A}>\frac{M}{A}+1$ and $\frac{M}{B}-1>\frac{M}{2B}$, by our second restriction $\zeta>1+C$ 
there are at most $\frac{\log(2\frac{M}{A})}{\log \zeta}-\frac{\log(\frac{M}{2B})}{\log \zeta}\leq
\frac{\log B-\log A+2\log 2}{\log (1+C)}\log M\asymp \log M$ intervals $I_{n,M_{n}}=I_{n,M}$
with $I_{n,M}\cap (A,B)\neq \emptyset$. Hence
\[
\lambda_{s}(\Delta_{\delta+\epsilon,A,B,C})\leq \sum_{M\geq M_{0}} C_{0}\log M\cdot 2^{s}M^{-(1+\delta+\frac{\epsilon}{2})s}
=2^{s}C_{0}\sum_{M\geq M_{0}} \log M\cdot M^{-(1+\delta+\frac{\epsilon}{2})s}
\]
for a constant $C_{0}$ and any $M_{0}\geq 1$. Since this sum converges for all $\epsilon>0$
qand $s\geq \frac{1}{1+\delta}$, the $s$-dimensional Hausdorff 
measure of $\Delta_{\delta+\epsilon,A,B,C}$ is arbitrarily small so it is $0$.
Assertion (\ref{eq:sixfootfour}) follows with Lemma \ref{einzigeslemma} putting $\Omega=\mathbb{R},\mu=\lambda_{s}$ and
$M_{\epsilon}:= \Delta_{\delta+\epsilon,A,B,C}$.

The specifications follow by the definition of Hausdorff dimension resp. $\delta=0$. 
\end{proof}

Now we want to study the reverse situation of fixed $\alpha$.

\begin{definition}
For any real $\alpha$ and $\delta\geq 0$, denote with $\Theta_{\delta}(\alpha)$
the set of real $\zeta>1$ such that $\overline{\sigma}(\alpha,\zeta)> \delta$.
In particular denote with $\Theta_{\infty}(\alpha)$ the set with $\overline{\sigma}(\alpha,\zeta)=\infty$.
\end{definition}

First an easy Proposition about the algebraic structure of the sets $\Theta_{\delta}(\alpha)$.

\begin{proposition}
 For $\delta_{1}<\delta_{2}$ we have $\Theta_{\delta_{1}}(\alpha)\supset \Theta_{\delta_{2}}(\alpha)$.
 If $\zeta\in{\Theta}_{\delta}(\alpha)$ then $N\sqrt[k]{\zeta}\in{\Theta_{\tau(\delta)}}$ 
with $\tau(\delta)=\max\left\{\frac{\delta\log \zeta -\log N}{\log \zeta+ \log N},0\right\}$.
In particular, $\Theta_{\delta}(\alpha)$ is closed under any map $\zeta\mapsto \sqrt[k]{\zeta}$ and
 $\Theta_{\infty}(\alpha)$ is closed under any map $\zeta\mapsto N\sqrt[k]{\zeta}$ 
for all positive integer pairs $N,k$. Moreover $\Theta_{\delta}(M\alpha)\supset \Theta_{\delta}(\alpha)$
 for all $\delta\in{(0,\infty]}$ and integers $M$.
\end{proposition}

\begin{proof}
The first point is obvious by the definition of the quantities $\Theta_{\delta}(\alpha)$.
The remaining assertions follow immediately from (\ref{eq:nanoleuchte}) and Proposition \ref{interessant}. 
\end{proof}

\begin{theorem}
For any real $\alpha$ and any $b>a\geq 1$, the set $\Theta_{\infty}(\alpha)\cap (a,b)$ has the same cardinality as $\mathbb{R}$.
 In particular $\Theta_{\infty}(\alpha)$ is dense in $(1,\infty)$.
\end{theorem}

\begin{proof} 
We recursively construct such $\zeta$ in dependence of $\alpha$ such that
\begin{equation} \label{eq:pelzjaeger}
 \Vert \alpha\zeta^{n_{u}}\Vert \leq \zeta^{-un_{u}}, \qquad u\geq 1
\end{equation}
for a monotonic increasing sequence $(n_{u})_{u\geq 1}$ we will specify.
This obviously yields $\overline{\sigma}(\alpha,\zeta)=\infty$.

Let $N_{1},n_{1}$ be arbitrary positive integers and put $\zeta_{1}:=\sqrt[n_{1}]{\frac{N_{1}}{\alpha}}$.
For reasons of continuity exists some interval $I_{1}:=(\zeta_{1}-\epsilon_{1},\zeta_{1}+\epsilon_{1})$
 such that $\Vert \alpha x^{n_{1}}\Vert \leq \zeta_{1}^{-n_{1}}$ for all $x\in{I_{1}}$.
So all $x\in{I_{1}}$ satisfy (\ref{eq:pelzjaeger}) for $j=1$.

Having defined intervals $I_{1}\supset I_{2}\supset I_{3}\cdots\supset I_{u}$ such
that (\ref{eq:pelzjaeger}) holds for all $x\in{I_{u}}$, we now define
 $(a_{u+1},b_{u+1})=I_{u+1}\subset I_{u}=(a_{u},b_{u})$ such that
it is still valid for $x\in{I_{u+1}}$.
Take an integer $n_{u+1}$ sufficiently large such that $\alpha(b_{u}^{n_{u+1}}-a_{u}^{n_{u+1}})> 2$.
Consequently the interval $(\alpha a_{u}^{n_{u+1}},\alpha b_{u}^{n_{u+1}})$ contains a neighborhood of $2$
 consecutive integers $N_{u+1,i(u)}, i(u)\in{\{1,2\}}$ respectively. 
Hence by continuity
\[
I_{u+1,i(u+1)}:=\left(\sqrt[n_{u+1}]{\frac{N_{u+1,i(u+1)}-\delta_{u+1}}{\alpha}},\sqrt[n_{u+1}]{\frac{N_{u+1,i(u+1)}+\delta_{u+1}}{\alpha}}\right)
\subset{I_{u}}
\]
for some $\delta_{u+1}>0$. 
If we decrease $\delta_{u+1}$ if necessary, for any $\zeta_{u+1}\in{I_{u+1,i(u+1)}}$
\begin{equation} \label{eq:willwex}
\vert \alpha\zeta^{n_{u+1}}-N_{u+1,i(u+1)}\vert \leq b_{1}^{-(u+1)}.
\end{equation}
For $u\geq u_{0}$ sufficiently large $b_{1}^{-(u+1)}<\frac{1}{2}$ so
 that the resulting intervals $I_{u+1,i(u+1)}$ are disjoint in each step.
So any $\zeta:=\cap_{u\geq 1} I_{u,i(u)}$ arising by this 
construction with any choice $i(u)\in{\{1,2\}}$ for $u=2,3,\ldots$ 
 satisfies (\ref{eq:pelzjaeger}) by (\ref{eq:willwex}) and $\zeta=\inf_{u\geq 1}b_{u}<\sup_{u\geq 1}b_{u}=b_{1}$,
 and no two such $\zeta$ coincide unless $i(u)$ coincide for all $u\geq u_{0}$.
Hence the set has cardinality of the power set of $\mathbb{N}$, which equals the cardinality of
 $\mathbb{R}$. This still holds with $\zeta$ restricted to $I_{1}$ as all arising $\zeta$ are in $I_{1}$.
Since the choice of $N_{1},n_{1}$ was arbitrary and clearly $\{\sqrt[n_{1}]{N_{1}}: (n_{1},N_{1})\in{\mathbb{N}^{2}}\}$
is dense in $(1,\infty)$ and we can make $\epsilon_{1}$ 
and hence $I_{1}$ smaller if needed, the assertion follows.
\end{proof}

The following Theorem \ref{willheim} shows that for fixed $\alpha\neq 0$, the set of $\zeta$
with good approximation properties concerning $\overline{\sigma}(\alpha,\zeta)$ is small
in measure theoretic sense.

\begin{theorem} \label{willheim}
For any $\delta\in{[0,\infty]}$ the $s$-dimensional Hausdorff measure of $\Theta_{\delta}(\alpha)$ is $0$
for all $s\geq \frac{1}{1+\delta}$. In particular: 
\begin{itemize}
\item the Hausdorff dimension of $\Theta_{\delta}(\alpha)$ is at most $\frac{1}{1+\delta}$
\item for any fixed $\alpha\neq 0$ the set of $\zeta>1$ with 
$\overline{\sigma}(\alpha,\zeta)>0$ has Lebesgue measure $0$.
\end{itemize}
\end{theorem}

\begin{proof}
We may restrict to $\alpha>0$ as $\alpha\mapsto -\alpha$ preserves the property.
In order to estimate the $s$-dimensional Hausdorff measure we look at the sets
$\Theta_{\delta,A}(\alpha)\subset \Theta_{\delta}(\alpha)$
defined as $\Theta_{\delta}(\alpha)$ with the additional restriction $\zeta>1+A$ for a parameter $A>0$.
If we can prove that the $s$-dimensional Hausdorff measure $\lambda_{s}$ of $\Theta_{\delta,A}(\alpha)$ equals $0$ 
for arbitrary $s\geq \frac{1}{1+\delta}$ and any $A>0$, then $\Theta_{\delta}(\alpha)$
has the same property by Lemma \ref{einzigeslemma} for 
$\Omega=\mathbb{R},\mu=\lambda_{s}$ and $M_{\epsilon}:= \Theta_{\delta,1+\epsilon}(\alpha)$.

So summing up we are left to prove 

\begin{equation} \label{eq:legalegali}
 \lambda_{s}(\Theta_{\delta,A}(\alpha))= 0, \qquad \qquad \forall s\geq \frac{1}{1+\delta}, \quad \forall A>0.
\end{equation}

By definition of $\overline{\sigma}(\alpha,\zeta)$ for $\zeta$ to be in $\Theta_{\delta}(\alpha)$
 we must have $\vert \alpha\zeta^{n}-M_{n}\vert \leq \zeta^{-n(\delta+\epsilon)}$
for some $\epsilon>0$ and arbitrarily large values of $n$ and integers $M_{n}$. 
Clearly $\zeta^{-n(\delta+\epsilon)}<1$ so $\frac{M_{n}-1}{\alpha}\leq \zeta^{n}\leq \frac{M_{n}+1}{\alpha}$
and clearly for sufficiently large $n$ we may replace 
the upper bound by $2\frac{M_{n}}{\alpha}$. Combining these facts gives
\[
 \zeta\in{\left( \sqrt[n]{\frac{M_{n}-\left(\frac{\alpha}{2M_{n}}\right)^{\delta+\epsilon}}{\alpha}},
 \sqrt[n]{\frac{M_{n}+\left(\frac{\alpha}{2M_{n}}\right)^{\delta+\epsilon}}{\alpha}} \right)}
\]
for arbitrarily large $n$ and corresponding integers $M_{n}$.
In particular with $\beta_{N,\epsilon}:= \left(\frac{\alpha}{2N}\right)^{\delta+\epsilon}$ 
any $\zeta\in{\Theta_{\delta+\epsilon}(\alpha)}$ satisfies
\[
\zeta\in{\bigcap_{n_{0}\geq 1}\bigcup_{\substack{n\geq n_{0},\\ N\geq 2}} J_{N,n}}, 
\qquad J_{N,n}:=\left(\alpha^{-\frac{1}{n}}\sqrt[n]{N-\beta_{N,\epsilon}},
 \alpha^{-\frac{1}{n}}\sqrt[n]{N+\beta_{N,\epsilon}}\right).
\]
Note that by our assumption $\zeta>1+A$ we have that
$n\to\infty$ is equivalent to $N\to\infty$, which we will implicitly
use in the sequel. By intermediate value Theorem of differentiation
we have that the interval $J_{N,n}$ has length at most 
$2\alpha^{-\frac{1}{n}}\beta_{N} \frac{1}{n}(N+\beta_{N,\epsilon})^{-1+\frac{1}{n}}$.
 Since $\beta_{N,\epsilon}\thicksim K\cdot N^{-\epsilon-\delta}$ with a constant $K$ as $N\to\infty$
and $\lim_{n\to\infty} \alpha^{-\frac{1}{n}}=1$, 
for every $s\geq \frac{1}{1+\delta}$ and sufficiently large $n=n(s)$
 we have the upper bound $C_{1}\cdot N^{-1-s\epsilon}n^{-s}$ with a constant $C_{1}>0$
for the $s$-th power of the interval length of $J_{N,n}$. So in particular
 for fixed $N$ we have $C_{1}\cdot N^{-1-s\epsilon}$
is a bound for the $s$-th power of all interval lengths $J_{N,n}$ that contribute to $\Theta_{\delta+\epsilon,A}(\alpha)$.

However, by the additional assumption $\zeta>1+A$ we have that for every $N$
there are at most $\frac{1}{\log (1+A)}\log N$ such intervals $J_{N,n}$.
So the $s$-dimensional Hausdorff measure of $\Theta_{\delta+\epsilon,A}(\alpha)$ is bounded above by
$\sum_{N\geq N_{0}} C_{2}\cdot N^{-1-s\epsilon}\log N$  for a constant $C_{2}$ and any $N_{0}$.
Since this sum converges we see that for any $\epsilon>0$
the $s$-dimensional Hausdorff measure of $\Theta_{\delta+\epsilon,A}(\alpha)$
is arbitrarily small, so it is $0$. Lemma \ref{einzigeslemma} with $\Omega=\mathbb{R},\mu=\lambda_{s}$ and
$M_{\epsilon}:=\Theta_{\delta+\epsilon,A}(\alpha)$ shows that the same holds for
the set $\Theta_{\delta,A}(\alpha)$ as well. Thus we have proved (\ref{eq:legalegali}).

The specifications of the Theorem follow by definition of Hausdorff dimension
respectively putting $\delta=0$.  
\end{proof}

\begin{remark}
 It is likely that analogously to Theorem \ref{ministerium} the bound in Theorem \ref{willheim} is uniform in $\alpha$.
Similarly as in the proof we could restrict to $\alpha\in{(c,d)}$ and
$1+A<\zeta<B$ to prove this using Lemma \ref{einzigeslemma}.
However, the straight forward construction above doesn't allow to establish this result, conversely
for any $\delta$ and any $0<c<d$ the method does not allow to prove
 $\Theta_{\delta,c,d}:=\cup_{\alpha\in{(c,d)}} \Theta_{\delta}(\alpha)$ has Lebesgue-measure $0$.

Indeed, even if we let the intervals $(N-\epsilon,N+\epsilon)$ around integers $N$ shrink to points $N$
then $\alpha\zeta^{n}=N$ with $\alpha\in{(c,d)}$ gives $\zeta^{n}\in{(\frac{N}{d},\frac{N}{c})}$.
As a result, trying to bound the Lebesgue-measure (i.e. $s=1$) as in the proof above
 there appear sums $\sum_{N\geq N_{0}} a_{N}$ with $a_{N}\geq \sqrt[n]{\rho N}-\sqrt[n]{\delta N}$
for $\delta= \frac{1}{d},\rho=\frac{1}{c}$ provided that $(1+A)^{n} < \delta N< \rho N < B^{n}$
for some integer $n$. Clearly there are infinitely many $N$ with this property. 

On the other hand, by intermediate value Theorem for derivatives we get 
\[
\sqrt[n]{\rho N}-\sqrt[n]{\delta N}\geq (\rho N)^{-1+\frac{1}{n}}(\rho N-\delta N)
=\frac{\rho-\delta}{\rho}\cdot (N\rho)^{\frac{1}{n}}.
\]
But $(N\rho)^{\frac{1}{n}}>1$ for all $N>\frac{1}{\rho}$. 
Hence $\limsup_{N\to\infty} a_{N}\geq \frac{\rho-\delta}{\rho}>0$, so $\sum_{N\geq 1} a_{N}$ is far from converging.
This argument shows that for any fixed $B>1+A>1$, $0<c<d$ and any $n_{0}$,
the Lebesgue measure of the set of $\zeta\in{(1+A,B)}$ such that $\sigma_{n}(\alpha,\zeta)=\infty$
for some $\alpha\in{(c,d)}$ (that may depend on $\zeta,n$) and some $n\geq n_{0}$
does not converge to $0$ as $n_{0}\to\infty$. This does not imply 
$\lambda_{1}(\Theta_{\delta,c,d})>0$ though for any $\delta$.
\end{remark}

The following Theorem \ref{algebfall} shows (in particular) that for almost all 
$\zeta>1$ the set of indices $n$ with values of $\sigma_{n}(\alpha,\zeta)$ 
exceeding $1$ by some fixed $\epsilon>0$ has asymptotic density $0$, with the
exceptional set consisting of very special algebraic and possibly transcendental numbers.
The proof relies heavily on Khinchin's Theorem \ref{khini} and Roth's Theorem \ref{roth}.
To simplify the proof we introduce a well-known measure theoretic fact in form of a Lemma.

\begin{lemma}[Elementary measure theory]  \label{tonycu}
Let $M\subset{\mathbb{R}}$ be a set with Lebesgue measure $0$
and $f:\mathbb{R}\mapsto \mathbb{R}$ be a Lipschitz continuous function.
Then $f(M):=\{f(m):m\in{M}\}$ is Lebesgue measurable and has Lebesgue measure $0$ as well. 
\end{lemma}  

See \cite{elst}.                

\begin{definition}[fd-property] \label{fd}
For fixed reals $\zeta,\alpha,\epsilon>0$ denote $(n_{j})_{j\geq 1}$ the increasing
integer sequence such that $\sigma_{n_{j}}(\alpha,\zeta)\geq 1+\epsilon$, where for
reasons of simplicity we omit the dependence in the notation, and
call $(n_{j})_{j\geq 1}$ the sequence associated to $\alpha,\zeta,\epsilon$.
We say $\zeta\in{\mathbb{R}}$ has the fd-property (finite difference property)
if $\liminf_{j\to\infty} n_{j+1}-n_{j}<\infty$ for some pair $\alpha\neq 0,\epsilon>0$. 
\end{definition}

\begin{theorem} \label{algebfall}
Lebesgue almost all $\zeta>1$ do not have the fd-property. An algebraic $\zeta>1$ has 
the fd-property if and only if it is of the form $\zeta=\sqrt[L]{\frac{p}{q}}$ 
for positive integers $L,p,q$. In the latter case,if $L$ was chosen minimal with this property,
 $\liminf_{j\to\infty} n_{j+1}-n_{j}\geq L$ for any pair $\alpha\neq 0,\epsilon>0$
and for any $\epsilon\in{(0,\infty]}$ the set of $\alpha$ with $\liminf_{j\to\infty} n_{j+1}-n_{j}=L$
has cardinality of $\mathbb{R}$.
\end{theorem}

\begin{proof}
Assume $\zeta>1$ has the fd-property, i.e. for some pair $\alpha\neq 0, \epsilon>0$
we have $\lim_{j\to\infty} n_{j+1}-n_{j}<\infty$ in the sense of Definition \ref{fd}. 
This means there exists a positive integer $d$ such that
$n_{j+1}-n_{j}=d$ for infinitely many $j$. By definition of $\overline{\sigma}(\alpha,\zeta)$
this yields
\begin{eqnarray*}
 \vert\alpha \zeta^{n_{j}}- M_{j}\vert &\leq& \zeta^{-(1+\epsilon)n_{j}} \\
 \vert \zeta^{d}\cdot \alpha\zeta^{n_{j}}- N_{j}\vert &\leq& \zeta^{-(1+\epsilon)(n_{j}+d)}
\end{eqnarray*}
has solutions $(M_{j},N_{j})\in{\mathbb{Z}^{2}}$ for arbitrarily large $n_{j}$.
Triangular inequality implies
\[
 \left\vert \zeta^{d}M_{j}-N_{j}\right\vert 
\leq \zeta^{d} \left\vert \alpha\zeta^{n_{j}}- M_{j}\right\vert + \left\vert \alpha\zeta^{n_{j}}- N_{j}\right\vert
\leq \zeta^{d-(1+\epsilon)n_{j}}+\zeta^{-(1+\epsilon)(n_{j}+d)}.
\] 
Since $d$ is fixed for sufficiently large $j$ clearly 
\[
\left\vert \zeta^{d}M_{j}-N_{j}\right\vert \leq \zeta^{-(1+\frac{\epsilon}{2})n_{j}}.
\]
Thus and recalling $M_{j}\thicksim \zeta^{n_{j}}$ we conclude 
\begin{equation} \label{eq:diegl}
\lambda_{1}(\zeta^{d})\geq 1+\frac{\epsilon}{2}.
\end {equation}
Now we distinguish several cases.
In case of algebraic $\zeta\neq \sqrt[d]{\frac{p}{q}}$ it follows $\zeta^{d}$ is non-rational algebraic,
so the above estimate contradicts Roth's Theorem \ref{roth}. Thus for $\zeta$ to have the fd-property
$\zeta=\sqrt[L]{\frac{p}{q}}$ is necessary. We show it is sufficient too.
In the special case of $\zeta=\sqrt[L]{M}$ for $M$ an integer, for $\alpha=1$ all the numbers $\alpha\zeta^{jL}$
are integers and hence $n_{j+1}-n_{j}\leq L$ for all $j\geq 1$ independently of $\epsilon>0$.
There actually must be equality for $j$ sufficiently large by Roth's Theorem and (\ref{eq:diegl}).
For $\zeta=\sqrt[L]{\frac{p}{q}}$ with relatively prime $p>q>1$, for arbitrary 
but fixed $\epsilon>0$ let $\alpha_{0}$ be any
number with $\overline{\sigma}(\alpha_{0},\zeta)=1+3\epsilon>1$. We know the cardinality of such numbers
equals the cardinality of $\mathbb{R}$ by Theorem \ref{kommtno}.  
Let $(n_{j})_{j\geq 1}$ be the sequence associated to $\alpha_{0},\zeta,3\epsilon$. 
Then $\alpha:=q\alpha_{0}$ has the property $\overline{\sigma}(\alpha,\zeta)=1+3\epsilon>1$
as well by Proposition \ref{interessant}, and its proof shows more precisely that
$q\scp{\alpha_{0}\zeta^{n_{j}}}=\scp{\alpha\zeta^{n_{j}}}$
and $\Vert \alpha\zeta^{n_{j}}\Vert \leq \zeta^{-(1+2\epsilon)n_{j}}$ for sufficiently large $j$
and the same sequence $(n_{j})_{j\geq 1}$ associated to $\alpha_{0},\zeta,3\epsilon$. 
Note that for any positive integer $N$ we obviously have $\Vert N\zeta\Vert\leq N\Vert \zeta\Vert$.
It follows that 
\begin{eqnarray*}
\left\Vert \alpha\zeta^{n_{j}+L}\right\Vert &=& \left\Vert q\alpha_{0}\zeta^{n_{j}+L}\right\Vert \\
&=& \left\Vert q\frac{p}{q} \alpha_{0}\zeta^{n_{j}}\right\Vert   \\
&=& \left\Vert p\alpha_{0}\zeta^{n_{j}}\right\Vert   \\
&\leq& p\left\Vert \alpha_{0}\zeta^{n_{j}}\right\Vert \leq p\zeta^{-(1+2\epsilon)n_{j}} \leq \zeta^{-(1+\epsilon)n_{j}}
\end{eqnarray*}
for sufficiently large $j$. So both $n_{j},n_{j}+L$ belong to the sequence associated to $\alpha,\zeta,\epsilon$,
thus $\liminf_{j\to\infty} n_{j+1}-n_{j}\leq L$ and in particular $\zeta$ has the fd-property. 
Finally again by Roth's Theorem and (\ref{eq:diegl}) we cannot have strict inequality, 
since this would give rise to some $d<L$ with $\zeta^{d}\in{\mathbb{Q}}$ contradicting the minimality of $L$. 
Putting everything together we have proved all the assertions for algebraic $\zeta$.

In general, in virtue of (\ref{eq:diegl})
$\zeta$ must be of the form $\sqrt[d]{\nu}$ for an integer $d$ and some real $\nu>1$ with $\lambda_{1}(\nu)>1$.
For any fixed $d\geq 1$ the map $\varphi_{d}:\mathbb{R}\mapsto \mathbb{R}$ defined by
 $\varphi_{d}(x)=\sqrt[d]{x}$ for $x>1$ and $\varphi_{d}(x)=1$ for $x\leq 1$ is clearly Lipschitz continuous,
hence by Khinchin's Theorem \ref{khini} and Lemma \ref{tonycu}
the set $\Omega_{d}:=\{\sqrt[d]{\nu}: \nu>1, \lambda_{1}(\nu)>1\}$ has Lebesgue measure $0$. 
By sigma subadditivity of measures the union $\cup_{d\geq 1} \Omega_{d}$ 
has Lebesgue measure $0$ as well, which finally proves the metric result.   
\end{proof}

\begin{remark}
In case of $\zeta=\sqrt[L]{M}, \alpha=1$ applying Proposition \ref{elendiglich} 
in fact shows that for any $\epsilon>0$ we have the asymptotics $n_{j}\thicksim jL$
as $j\to\infty$, provided that $L$ was chosen minimal
with the property $\zeta=\sqrt[L]{r}$ for $r\in{\mathbb{Q}}$.
\end{remark}

\begin{remark}
Note that since Theorem \ref{willheim} is not uniform in $\alpha$, it does not
trivially imply the metric result of Theorem \ref{algebfall}.
\end{remark}

Note that Theorem \ref{algebfall} in particular
applies to the assumptions of section \ref{iii}, i.e. if $\alpha$ is algebraic too.
However, in this case it would obviously be implied by Conjecture \ref{confituer} in a trivial way.

\vspace{1cm}

Johannes Schleischitz  \\
Anzengrubergasse 23/2, 1050 Vienna

\end{document}